\documentclass{article}
\usepackage[utf8]{inputenc}
\usepackage[small,compact]{titlesec}
\usepackage{amsmath,amssymb,amsfonts,mathabx,mathrsfs,setspace}
\usepackage{graphicx,caption,epsfig,subfigure,epsfig,wrapfig,sidecap}
\usepackage{url,color,verbatim,algorithmicx}
\usepackage[sort,compress]{cite}
\usepackage{soul}
\usepackage{bm,float}
\usepackage[top=1in,bottom=1in,left=1in,right=1in]{geometry}
\usepackage{multirow, makecell}
\usepackage[ruled,boxed]{algorithm}
\usepackage[scaled]{helvet}
\usepackage[T1]{fontenc}
\usepackage[bookmarks=true, bookmarksnumbered=true, colorlinks=true,   pdfstartview=FitV,
linkcolor=blue, citecolor=blue, urlcolor=blue]{hyperref}
\usepackage{multirow, multicol,makecell, colortbl, hhline}
\usepackage{booktabs} 
\usepackage{enumitem} 
\usepackage{etoc}          

\def\Gap{\mathrm{Gap}}
\def\deltat{\delta}

\usepackage[noend]{algpseudocode}

\newcommand{\argmin}[1]{\underset{#1}{\operatorname{arg}\operatorname{min}}\;}

\newtheorem{theorem}{Theorem}
\newtheorem{assumption}[theorem]{Assumption}

\newtheorem{proposition}[theorem]{Proposition}
\newtheorem{remark}[theorem]{Remark}
\newenvironment{proof}[1][Proof]{\noindent\textbf{#1.} }{\ \rule{0.5em}{0.5em}}

\numberwithin{equation}{section}
\numberwithin{theorem}{section}

\def\calE{\mathcal{E}}
\newcommand{\prob}[1]{\mathbb{P}\left(#1\right)}

\newcommand{\R}{\mathbb{R}}
{}

\newcommand{\mcE}{\mathcal{E}}

\newcommand{\mcF}{\mathcal{F}}

\usepackage{xcolor}

\def \Xb{\mathbf{X}}

\def \E{\mathbb{E}}

 \def\mW{{\bm{W}}}

\def\mX{{\bf X}}     
\def\mY{{\bf Y}}

\def\mJ{{\bf J}}

\def\my{{\bf y}}

\newenvironment{rmat}{\left[\begin{array}{rrrrrrrrrrrrr}}{\end{array}\right]}
\newcommand\brm{\begin{rmat}}
\newcommand\erm{\end{rmat}}
\newenvironment{cmat}{\left[\begin{array}{ccccccccc}}{\end{array}\right]}
\newcommand\bcm{\begin{cmat}}
\newcommand\ecm{\end{cmat}}

\newcommand{\mbf}[1]{\bm{#1}}

\title{NySALT: Nystr\"{o}m-type inference-based schemes adaptive to large time-stepping
\author{Xingjie Li, Fei Lu, Molei Tao, Felix Ye }
}

\begin{document}
\maketitle
\begin{abstract}
Large time-stepping is important for efficient long-time simulations of deterministic and stochastic Hamiltonian dynamical systems. Conventional structure-preserving integrators, while being successful for generic systems, have limited tolerance to time step size due to stability and accuracy constraints.   
We propose to use data to innovate classical integrators so that they can be adaptive to large time-stepping and are tailored to each specific system. In particular, we introduce 
NySALT, Nystr\"{o}m-type inference-based schemes adaptive to large time-stepping. The NySALT has optimal parameters for each time step learnt from data by minimizing the one-step prediction error. Thus, it is tailored for each time step size and the specific system to achieve optimal performance and tolerate large time-stepping in an adaptive fashion. We prove and numerically verify the convergence of the estimators as data size increases. Furthermore, analysis and numerical tests on the deterministic and stochastic Fermi-Pasta-Ulam (FPU) models show that NySALT enlarges the maximal admissible step size of linear stability, and quadruples the time step size of the St\"{o}rmer--Verlet and the BAOAB 
when maintaining similar levels of accuracy. 
\end{abstract}
\vspace{-2 mm}
\vspace{- 2mm}
\tableofcontents
\section{Introduction}
Efficient simulations of Hamiltonian dynamical systems and their stochastic generalizations play an essential role in many applications where the goal is to capture and predict both short time and long time dynamics. 
Conventional structure-preserving integrators have achieved tremendous success in 
 preserving structures such as symplecticity, reversibility, manifold structure, other physical constraints, and even statistical properties in long-time simulations (see \cite{sanz-serna_1992,calvo1994numerical, ASCHER1997,marsden_west_2001,DaSilva2003,Leimkuhler2004,Hairer2006,Tao2010,bou2010long, ObTa13, tao2016explicit, tao2016explicitJCP, blanes2017concise, tao2020variational, chen2021grit, tao2022accurate, leimkuhler2013rational, Feng2010_book,Hong2006,Hong2019_book} and the references therein). Due to their remarkable suitability for Hamiltonian systems, this article will focus on symplectic integrators.
 
Meanwhile, the time step of generic (symplectic) integrators are often limited by the stiffness of these systems, making the simulation computationally costly. Also, the conventional integrators aim for generic systems, not taking into account each specific Hamiltonian. We aim at using data to construct  \emph{large time-stepping integrators} 
 that can tolerate large time steps while maintaining symplecticity, stability, and accuracy, and more importantly, \emph{tailored to each specific Hamiltonian} in an automatic fashion.

It is a fast growing research area to leverage information from data and combine statistical tools 
with traditional scientific computing methodology.  
This work is under this umbrella as we utilize statistical learning tools to approximate a discrete-time flow map and then construct large time-stepping integrators.

In this paper, we propose to construct large time-stepping integrators by inferring optimal parameters of classical structure-preserving integrators in a flow map approximation framework. The inferred schemes are adaptive to the time-step size, thus they can have large time-stepping while maintaining stability and accuracy. Furthermore, the parameters are low-dimensional and can be learned from limited data consisting of short trajectories, and the estimator converges as the data size increases under suitable conditions. Consequently, the inferred integrators are robust and generalizable beyond the training set (see Section \ref{sec:framework}). A few work of similar spirit can be found in  \cite{SOTA_Tao2021,zheng2021learning,liu2022seven}, 
where neural network based approximations lead to large time-stepping integrators. 
  However, the neural networks are computationally expensive to train and their parameters are often sensitive to training data, making it difficult to systematically investigate its properties such as the maximal admissible time step size of stability. 
Besides, most designs of neural networks are disconnected from the classical numerical integrators.

For benchmark application, we focus on parametric integrators in the Nystr\"{o}m family (see descriptions from e.g., \cite{Hairer2006}), which includes the popularly used St\"{o}rmer--Verlet method. From observed data, we then construct the NySALT scheme: Nystr\"{o}m-type inference-based scheme adaptive to large time-stepping (NySALT). The NySALT ensures optimal parameters for each time step 
by minimizing the one-step prediction error learnt from data. {Linear stability analysis is} also established to verify our premise, namely that the optimal parameters indeed should be different from those of the St\"{o}rmer--Verlet, and the resulting NySALT has a larger maximal admissible step size for linear stability (see Section \ref{sec:linearAnalysis}). We examine the performance of NySALT on the widely-used benchmark stiff nonlinear systems: a deterministic Fermi-Pasta-Ulam (FPU) model, as well as its stochastic (Langevin) generalization.  Numerical results show that the inference of NySALT is robust: the estimators are independent of the fine data generators, they converge as the number of trajectories (size of observed data) increases, and they stabilize very fast (within a dozens of short trajectories). 
It also shows the NySALT is accurate: it is adaptive to large time step size, and improves the accuracy of trajectories in multiple time scales and statistics in long time scale. 
Lastly, the NySALT is efficient: it enlarges the admissible time step size of stability of the classical schemes such as the St\"{o}rmer--Verlet and the BAOAB methods \cite{Leimkuhler2004,Hairer2006,leimkuhler2013rational} (see Section \ref{sec:numFPU}) and significantly reduces the simulation time.

Our main contributions are threefold. \vspace{-2mm}
\begin{itemize}\setlength\itemsep{-0mm}
\item We propose to infer the large time-stepping and structure-preserving integrator from data in a flow-map approximation framework, in which we select optimal parameters in a family of classical geometric numerical integrators by minimizing the flow map approximation error.  If we choose the Nystr\"{o}m family, it is NySALT scheme.
\item  The inference procedure of NySALT is robust and the scheme is generalizable beyond the training set.   
\item Analysis and numerical tests show that NySALT scheme is efficient and accurate with large time step size.
\end{itemize}
Meanwhile, many work that employs data-driven approaches in the past has tackled parts of our goal, but not all of them. These related work are usually categorized based on their models or methods and we summarize them here:
\begin{itemize}
    \item { \textbf {Learning large time-stepping integrators.}} 
It is an emerging research direction to learn large time-stepping integrators from data. When the system is known, MDNet based on graph neural network in \cite{zheng2021learning} enables the simulation of microcanonical (i.e. Hamiltonian) molecular dynamics with large steps; a stochastic collocation method in \cite{liu2022seven} and a parametric inference in \cite{ISALT2022} have lead to large time-stepping integrators for SDEs. This study extends the parametric inference approach in \cite{ISALT2022} to Hamiltonian systems. When the system is unknown, generating function neural network (GFNN) in \cite{SOTA_Tao2021} learns  
symplectic maps and proves a significant benefit of doing so, namely a linearly growing bound of long time prediction error.

\item {\textbf {Learning the Hamiltonian or the system.} }  
A very active research area is to recover the dynamics that generate observed, discrete time-series data, and then use the learned dynamics to predict further evolutions. Examples of existing work for Hamiltonian systems include 
\cite{greydanus2019hamiltonian, bertalan2019learning, chen2019symplectic,lutter2018deep,toth2019hamiltonian,zhong2019symplectic,jin2020sympnets,xiong2021nonseparable,SOTA_Tao2021,valperga2022learning}, and while early seminal work learned Hamiltonian vector fields without truly preserving symplecticity, later results leveraged various tools including symplectic integrator \cite{chen2019symplectic}, composition of triangular maps \cite{jin2020sympnets}, and generating function \cite{SOTA_Tao2021} to fix this imperfection. 
The setup of all these research, however, assumes that the governing dynamics is unknown (i.e. `latent' in machine learning terminology), which is different from our setup as we instead seek a good numerical integrator for given Hamiltonian.

\item {\textbf {Integrators of multiscale Hamiltonian systems.}}
There have been remarkable progress in generic upscaled integration of stiff and multiscale ODE systems (e.g., 
\cite{kevrekidis2004equation,weinan2007heterogeneous, abdulle2012heterogeneous, Ariel:08, calvo2010heterogeneous, Tao2010,Jin2022}) 
and despite that fewer results exist when it comes to generic multiscale symplectic integrators (e.g., \cite{Tao2010}), multiscale symplectic integrators for specific classes of problems have also been constructed (e.g., \cite{Grubmuller:91, Tuckerman:92, Skeel:99, LeBris:07, DoLeLe10, MR1490094, SIM2}).
While each of these integrators is tremendously useful for a specific class of systems, a complete re-design is likely necessary when a system is outside of the class. Our integrator, in contrast, is tailored to each specific Hamiltonian automatically as the outcome of the inference procedure.

\item {\textbf{Model reduction and time series modeling.}} Large time-stepping schemes can also be viewed as a model reduction in time for the differential equations (DEs).  A more challenging task is model reduction in both space-time, i.e., reducing the spatial dimension and integrating with large time-steps, for high-dimensional DEs or PDEs. This is an extremely active research area (see e.g., \cite{weinan2007heterogeneous,KMK03,legoll2010effective,MH13,CL15,leiDatadrivenParameterization2016,hudson2020coarse,Lu20Burgers,snyder2022reduced,li2017computing} and the references therein for a small sample of the important works). While it is impossible to review all important works, we mention the proper symplectic decomposition with Galerkin projection methods for Hamiltonian systems \cite{Peng2016,Hesthaven2017,Buchfink19}; the time series approaches (see e.g., \cite{KCG15,LLC17,LinLu21}) and the deep learning methods that solve PDEs (see e.g., \cite{bar2019learning,ma2018model}).

\end{itemize}

\section{Hamiltonian systems and parametric symplectic integrators}
We briefly review a few preliminary concepts of Hamiltonian systems and symplectic integrators. The classical symplectic integrators are designed as universal methods for all systems and they require a small time-step size to be accurate. To tolerate large time-step sizes, we will infer adaptive symplectic integrators from data, thus, we focus on symplectic schemes with parameters, which can be optimized by statistical inference from data.

\subsection{Hamiltonian systems and symplectic maps}
Let $H(\mbf{p},\mbf{q})$ be a Hamiltonian function on $\R^d \times \R^d$, where $\mbf{p}$ is the momentum and $\mbf{q}$ denotes the position. Consider the Hamiltonian ODE system
\begin{equation}\label{eq:Hamil_ode}
\begin{aligned}
\begin{cases}
d\mbf{q}(t)=\frac{\partial H}{\partial \mbf{p}}dt,\\
d\mbf{p}(t)=\frac{-\partial H}{\partial \mbf{q}}dt,
\end{cases}
\end{aligned}
\end{equation}
Let $\mX= (\mbf{q},\mbf{p})$ and let $\mJ = \left(\begin{array}{cc}\mbf{0}&\mbf{I}\\ -\mbf{I} & \mbf{0}\end{array} \right)$ be a $2d\times 2d$ matrix with $\mbf{0}$ and $\mbf{I}$ being $d\times d$ block matrices.  
We can write the Hamiltonian system as $
\mX(t) = \mJ \nabla_\mX H dt$. 
 A \emph{symplectic map} is a differentiable map $\phi:\R^{2d}\to \R^{2d}$ whose Jacobian matrix $\nabla \phi$ satisfies 
$\nabla \phi(\mX)^\top \mbf{J} \nabla \phi(\mX) = \mJ, \quad  \forall \mX \in \R^{2d}$.

Hamiltonian systems have a characteristic property: the flow of a Hamiltonian system is a symplectic map. More precisely, 
let $f:U\to \R^{2d}$ be a continuous differential function from an open set $U\subset\R^{2d}$. Then, $d\my=f(\my)dt$ is locally Hamiltonian if and only if its flow $\phi_t(\my)$ is symplectic for all $\my\in U$ and for all sufficiently small $t$ (\cite[Theorem 2.6, page 185]{Hairer2006}).
Furthermore, the flow is symplectic for any $t$ if the ODE is Hamiltonian as long as the flow is well posed (see e.g., \cite{abraham2008foundations, arnol2013mathematical})

This property is the starting point of our data-driven construction of numerical flows: we seek flows that are symplectic and are described by parameters to be estimated from data. 
There are various parametric families of symplectic integrators (see \cite{Hairer2006}). We consider in this study the Nystr\"{o}m family, which provides an explicit time integrator with two parameters. In particular, the widely-used St\"ormer--Verlet method is a member in this family.

In this study, we focus on the systems with separable Hamiltonian $H=K(\mbf{p})+V(\mbf{q})$ with $K(\mbf{p}) = \frac{1}{2}\|\mbf{p}\|^2$ and $V(\mbf{q})$ satisfying $g(\mbf{q}) =- \nabla V(\mbf{q})$, i.e., systems in the form
\begin{equation}\label{eq:Hamil_pq}
\begin{aligned}
\begin{cases}
d\mbf{q}(t)=\mbf{p}dt,\\
d\mbf{p}(t)=g(\mbf{q})dt . 
\end{cases}
\end{aligned}
\end{equation}
We will also consider damped stochastic Hamiltonian systems, which is also called Langevin dynamics:
\begin{equation}\label{eq:Hamil_pqstoc}
\begin{aligned}
\begin{cases}
d\mbf{q}(t)
=\mbf{p} dt,\\
d\mbf{p}(t)=(g(\mbf{q})  - \gamma \mbf{p})dt + \sigma d{\mbf{W}_t},
\end{cases}
\end{aligned}
\end{equation}
where $\gamma$ represents the friction coefficient, and $\mbf{W}_t$ is a standard (multi-dimensional) Wiener process.

\subsection{Deterministic Symplectic Nystr\"om scheme} \label{sec:DetNstrom}
Firstly, let us recall the $s$-step Nystr\"{o}m method for the second order differential equation \eqref{eq:Hamil_pq} following the notations from \cite{Hairer2006}. The Nystr\"{o}m method advances the dynamics from $t_0$ to $t_1$ with step size $h$ as 
\begin{align}\label{general_Nystrom}
\begin{cases}
\ell_{i}& = g\big(\mbf{q}_n+c_i h\mbf{p}_n+h^2\sum_{j=1}^{s}a_{ij}\ell_j \big), \quad i=1,\dots,s,\\
 \mbf{q}_{n+1} &= \mbf{q}_n +h\mbf{p}_n+h^2 \sum_{i=1}^{s} \beta_i\ell_i,\\
\mbf{p}_{n+1} &= \mbf{p}_{n} +h\sum_{i=1}^{s}b_i \ell_i,
\end{cases}
\end{align}
where $\{a_{ij}\}_{i,j=1}^s$, $\{c_i\}_{i=1}^{s}$, $\{\beta_i\}_{i=1}^{s}$ and $\{b_i\}_{i=1}^{s}$ are parameters to be specified.  
To have an explicit scheme, it requires that $\ell_i$ only depends on $\ell_j$ with $j<i$, hence $a_{ij}=0$ when $j\ge i$. 

We focus on the explicit $2$-step Nystr\"{o}m methods, which includes the widely-used St\"{o}rmer--Verlet method \cite{Hairer2006}.  We denote it by $S_{b_1,\beta_1}^h$: 
\begin{equation}\label{eq:nystrom2}
 \begin{aligned}
\left[   \begin{array}{c}
\mbf{q}_{n+1} \\
\mbf{p}_{n+1}
  \end{array}   \right]
& = S_{b_1,\beta_1}^h\left( \left[ \begin{array}{c}
\mbf{q}_{n} \\
\mbf{p}_{n}
  \end{array} \right]\right) 
= \left[\begin{array}{c}
  \mbf{q}_n+h \mbf{p}_n+h^2 (\beta_1\ell_1  +\beta_2\ell_2), \\
  \mbf{p}_n+h (b_1 \ell_1+b_2\ell_2)
  \end{array}      \right].
  \end{aligned}
\end{equation}
where
\begin{equation}\label{eq:l1l2n}
 \begin{aligned}
\ell_1 = g(\mbf{q}_n+ h c_1 \mbf{p}_n )  \quad \text{and} \quad \ell_2 = g(\mbf{q}_n+ h c_2 \mbf{p}_n + h^2 a_{21}\ell_1).  \end{aligned}
\end{equation}

Meanwhile,  by applying Taylor expansion to $\mbf{q}_n$ and $\mbf{p}_n$ in the second and third equations in \eqref{general_Nystrom} at $t_0$, we get consistency constraints on parameters $\{\beta_i\}$ and $\{b_i\}$:
\begin{equation}\label{Nystrom_constraint_eq1}
    \sum_{i=1}^{2}\beta_i=\frac{1}{2}\quad \text{and}\quad     \sum_{i=1}^{2}b_i=1.
\end{equation}
Notice that the general Nystr\"{o}m is not necessarily a structure-preserving scheme, so we need additional constraints on parameters to possess the symplectic property. From [Theorem 2.5 in Chapter IV \cite{Hairer2006}], a sufficient condition is:
\begin{equation}\label{Nystrom_constraint_eq2}
    \begin{split}
    \beta_i&= b_i (1-c_i)\quad \text{for}\quad i=1,2, \\
    b_i(\beta_j-a_{ij})&=b_j(\beta_i-a_{ji}) \quad \text{for}\quad i,j=1,2. 
    \end{split}
\end{equation}
Combining the constraints $a_{ij}=0$ for $j\ge i$, \eqref{Nystrom_constraint_eq1}--\eqref{Nystrom_constraint_eq2}, we have the following conditions on parameters to have an explicit and symplectic $2$-step Nystr\"{o}m scheme: 
\begin{equation}\label{Nystrom_constraint_summary}
    \begin{split}
        &\text{free parameters:}\quad 0<b_1<1,\; 0\le \beta_1\le \frac{1}{2},\\
        &\quad b_1+b_2=1,\quad \beta_{1}+\beta_2=\frac{1}{2},  \qquad c_i=1-\frac{\beta_i}{b_i}\quad \text{for}\quad i=1,2,\\
        &\quad a_{11}=a_{12}=a_{22}=0,\quad a_{21}=b_1(c_2-c_1).
            \end{split}
\end{equation} 
Notice that the well-known St\"{o}rmer--Verlet method belongs to this category with $b_1=1/2$ and $\beta_1=1/2$.
The above explicit and symplectic $2$-step Nystr\"{o}m integrator is of second order accuracy $O(h^2)$. Our NySALT scheme is the symplectic $2$-step Nystr\"{o}m integrator with the optimal parameters $b_1^*$ and $\beta_1^*$, which are learnt from data by minimizing the one-step prediction error (details see Section \ref{sec:flowDet}). 

\paragraph{Limited time step size of a classical numerical integrator.}
A major efficiency bottleneck of the majority of explicit numerical integrators is the limited time step size when the system is stiff.
As an illustration, we consider the the St{\"{o}}rmer--Verlet method and our benchmark example, the FPU model \eqref{eq:FPU-Hamiltonian}. By only considering the quadratic terms (i.e., the stiff harmonic oscillators) in the Hamiltonian, we obtain a linear stability condition on $h$ of this numerical integrator: $|\frac{h\omega }{2}|<1$, or equivalently $h<2/\omega$. Importantly, the accuracy of the integrator deteriorates quickly as $h$ increases, even when it is just half of the stability constraint. 
Figure~\ref{Fig:FPU_linear_stable} demonstrates this issue of the St{\"{o}}rmer--Verlet method, using the FPU model with $m=3$ and $\omega=50$ in the time interval  $[0,\, 500]$. It show the trajectories of the stiff energies $(I_j, j=1,\ldots, 3)$ and the total stiff energy $I$ defined in \eqref{eq:stiff_energy} from the St{\"{o}}rmer--Verlet integrator with two time step sizes, fine time step size $h=1\mathrm{e}{-4}$ and coarse time step size $\delta =0.02$, which is 200 times of fine time step size. Clearly, the method becomes inaccurate when $\delta=0.02$, which is still within the linear stability region $\delta<2/\omega=0.04$. In comparison, our NySALT scheme with the coarse time step size still performs as accurate as the one with fine step size. Particularly, the total stiff energy $I$ are well conserved over long time interval. More detailed analysis is shown in Section \ref{sec:detFPU_num}.

\begin{figure}[htp!]
    \centering
    \subfigure[$h=1\mathrm{e}{-4}$ St{\"{o}}rmer--Verlet]{\includegraphics[width =0.32\textwidth]{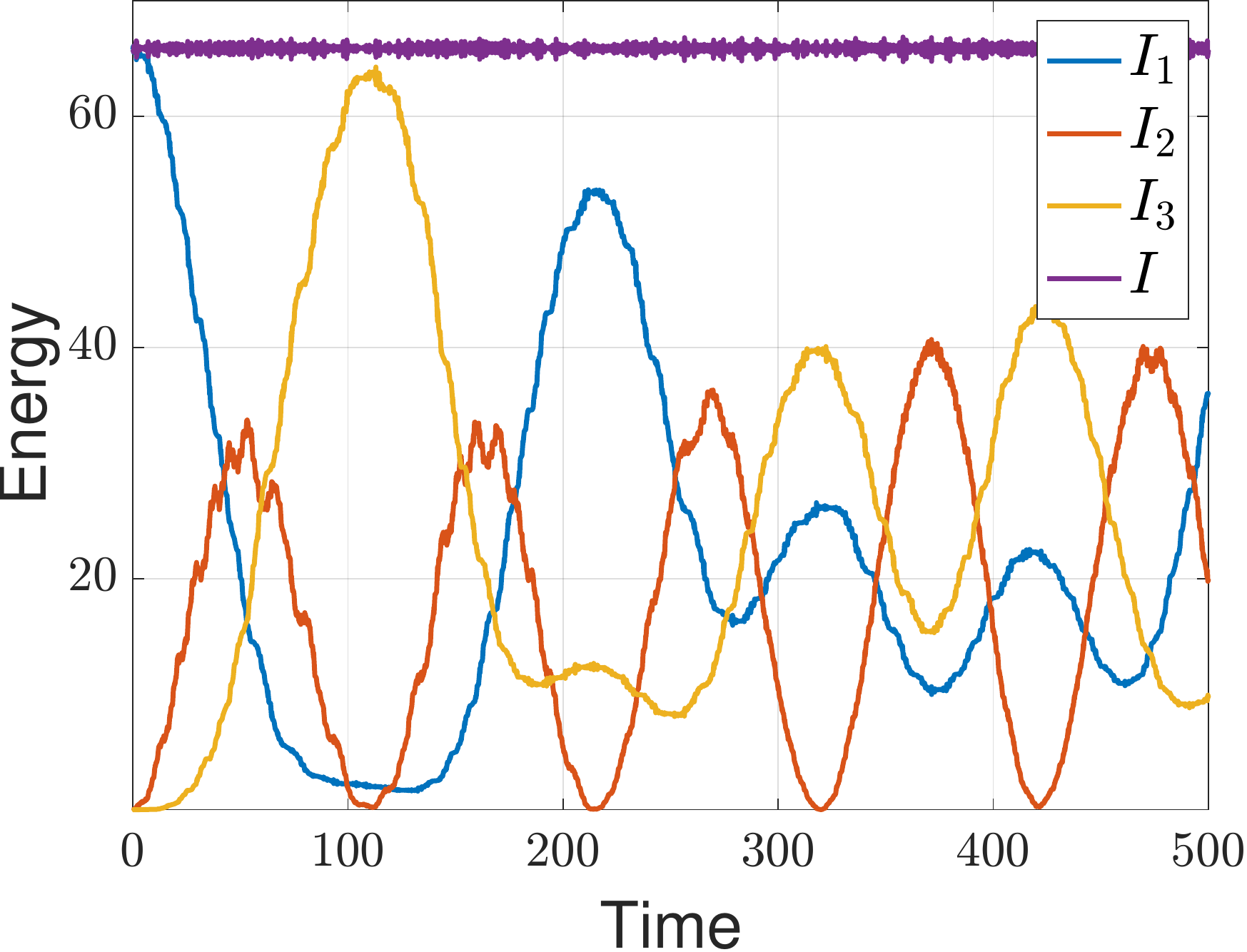}}
    \subfigure[$\delta=0.02$ St{\"{o}}rmer--Verlet]{\includegraphics[width =0.32\textwidth]{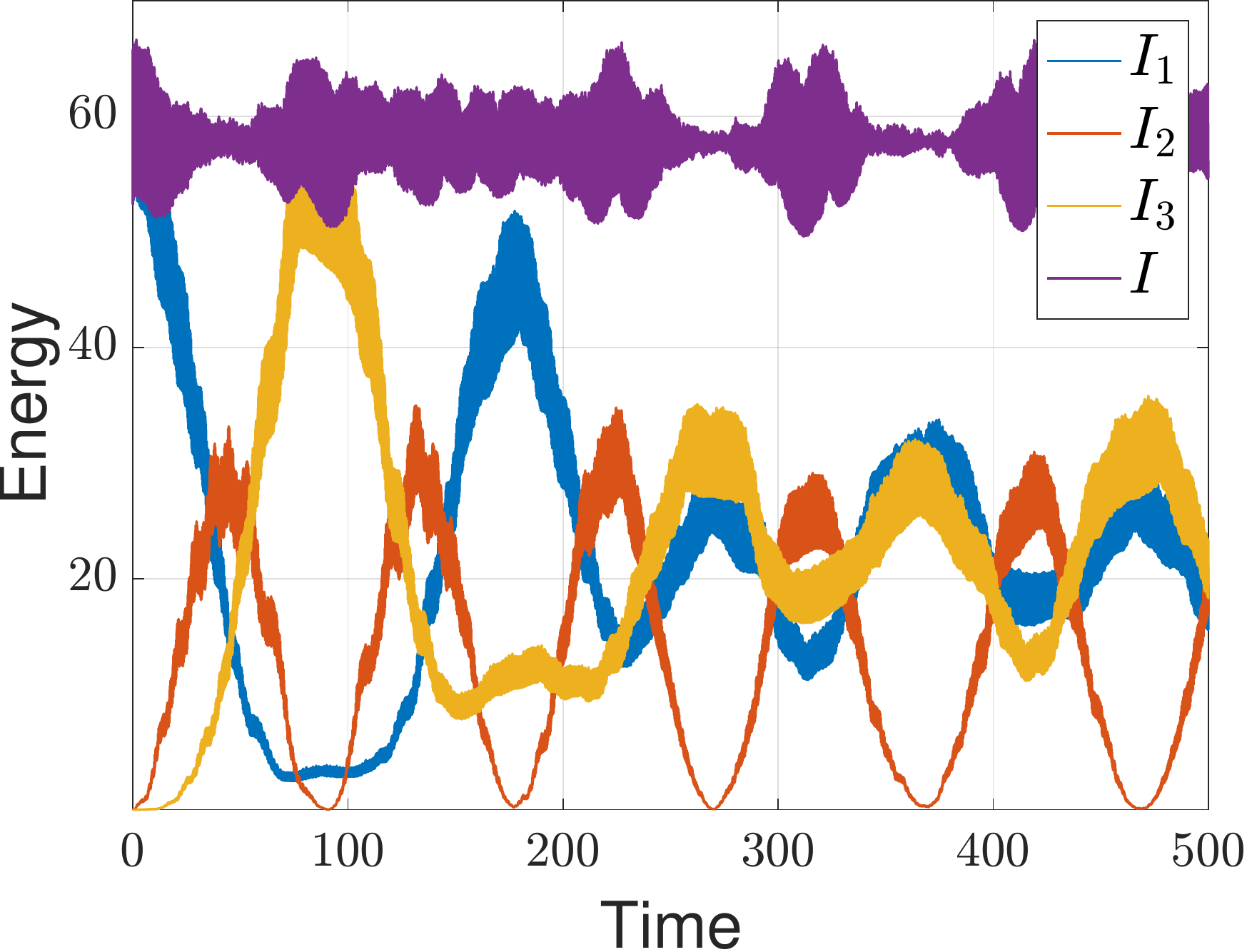}}
    \subfigure[$\delta=0.02$ NySALT]{\includegraphics[width =0.32\textwidth]{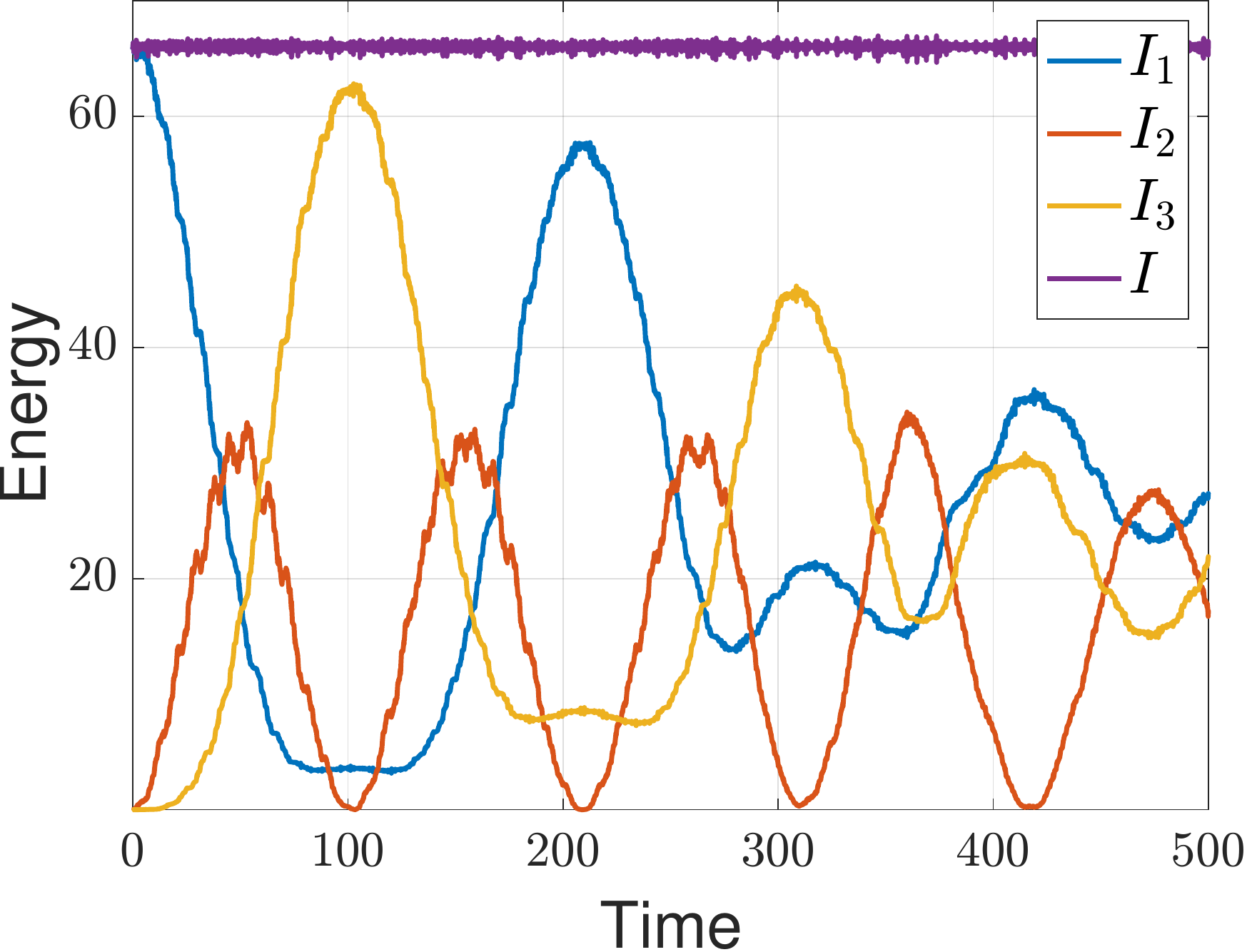}}
    \caption{Trajectories of each energy $I_j$ as well as the total stiff energy $I$ \eqref{eq:stiff_energy} of the deterministic FPU \eqref{eq:FPU-Hamiltonian} with $m=3$ and $\omega=50$, computed by the St{\"{o}}rmer--Verlet scheme with time step sizes $h=1\mathrm{e}{-4}$ in (a) and $\delta=0.02$ in (b), and by the NySALT scheme with $\delta = 0.02$ in (c). 
    }\label{Fig:FPU_linear_stable}
\end{figure}

\subsection{Stochastic Symplectic Nystr\"{o}m scheme}
To have a parametric scheme for the Langevin dynamics, we introduce a new splitting scheme by combining the Nystr{\"{o}}m integrator with an Ornstein–Uhlenbeck process, and we call it a stochastic Symplectic Nystr{\"{o}}m  scheme.  This scheme is the splitting methods (e.g. \cite{mclachlan2002splitting,bou2010long}) and similar to the BAOAB and ABOBA schemes (e.g., \cite{leimkuhler2013rational,Xiaocheng2021}). We break the Langevin dynamics into two pieces, the Hamiltonian part and the Ornstein–Uhlenbeck (OU) process part. 
\begin{equation}\label{eq:splitting}
\begin{aligned}
\bcm d\mbf{q}(t) \\ d\mbf{p}(t) \ecm 
= \bcm \mbf{p} \\ g(\mbf{q}) \ecm dt + \bcm 0 \\ -\gamma \mbf{p} dt + \sigma d{\mbf{W}_t} \ecm
\end{aligned}
\end{equation}
Each of them are solved separately as follows: 
it combines the symplectic Nystr{\"{o}}m approximation of the Hamiltonian contribution and an Ornstein–Uhlenbeck (OU) integrator approximation of the friction and thermal diffusion of the system. Given a time step $h$, this scheme reads  
\begin{equation}\label{eq:SO_def}
\begin{aligned}
\text{Deterministic Symplectic Nystr\"om scheme: }& \left[
 \begin{array}{c}
 \mbf{q}_{n+1}\\  {\tilde{\mbf{p}}}_{n+1}
 \end{array}\right]  =  S^{h}_{b_1,\beta_1}\left(\left[
 \begin{array}{c}
 \mbf{q}_{n}\\  \mbf{p}_{n}
 \end{array}\right]\right)\\
\text{Ornstein–Uhlenbeck integrator: }&{\mbf{p}}_{n+1} =  \exp(-\gamma h){\tilde{
\mbf{p}}}_{n+1} + \xi_n, 
\end{aligned}
\end{equation}
where $\{\xi_n\}$ is a sequence of independent identically distributed Gaussian vectors with distribution $\mathcal{N}(0, \frac{\sigma^2}{2\gamma} (1- e^{-2\gamma h}  ) \mbf{I}_d)$. The symplectic integrator $S^{h}_{b_1,\beta_1}$ is the 2-step Nystr\"om integrator in \eqref{eq:nystrom2}, thus it preserves the Hamiltonian contribution in the stochastic system and enhances the numerical stability. The second part for the stochastic force is based on the exact solution of the OU process and it leads to the local error of order $O(h^{1.5})$ in $\mathbf{q}$. Similar to the deterministic case, this  stochastic Symplectic Nystr{\"{o}}m  scheme is a family of numerical schemes and  
our NySALT scheme is the one with the optimal parameters $b_1^*$ and $\beta_1^*$, which are learnt from data by minimizing the one-step prediction error (details see Section \ref{sec:flowSto}).  
Thus, the NySALT scheme is of local strong order $h^{1.5}$ (see Remark \ref{rmk:orderSNO} for a derivation for the linear system and we refer to for instance \cite{telatovich2017strong} for a thorough study on the strong order of splitting schemes for Langevin dynamics.)

\paragraph{Limited time step size of a classical numerical integrator.} Similar to the deterministic systems, numerical integrators for stochastic systems can tolerate limited time step size. To demonstrate it, we consider Langevin dynamics with FPU potential and we choose the friction coefficient $\gamma =0.01$ and diffusion coefficient $\sigma=0.05$. The total energy $I$ in this example is stochastic, so we calculate its time auto-covariance function (ACF). Figure~\ref{Fig:SFPU_gap200} shows the ACF computed by BAOAB scheme\cite{leimkuhler2013rational}, one of the state-of-the-art sympletic integrator, in comparison with our NySALT scheme , both using the coarse step size $\delta = 0.02$. The reference is computed by BAOAB with fine step sizes $h=1\mathrm{e}{-4}$. As can be seen, the BAOAB scheme produces inaccurate ACF, while the NySALT scheme remains very reliable. More detailed analysis is in Section \ref{sec:stocFPU_num}. 

\begin{figure}[htp!]
    \centering
    \includegraphics[width =0.45\textwidth]{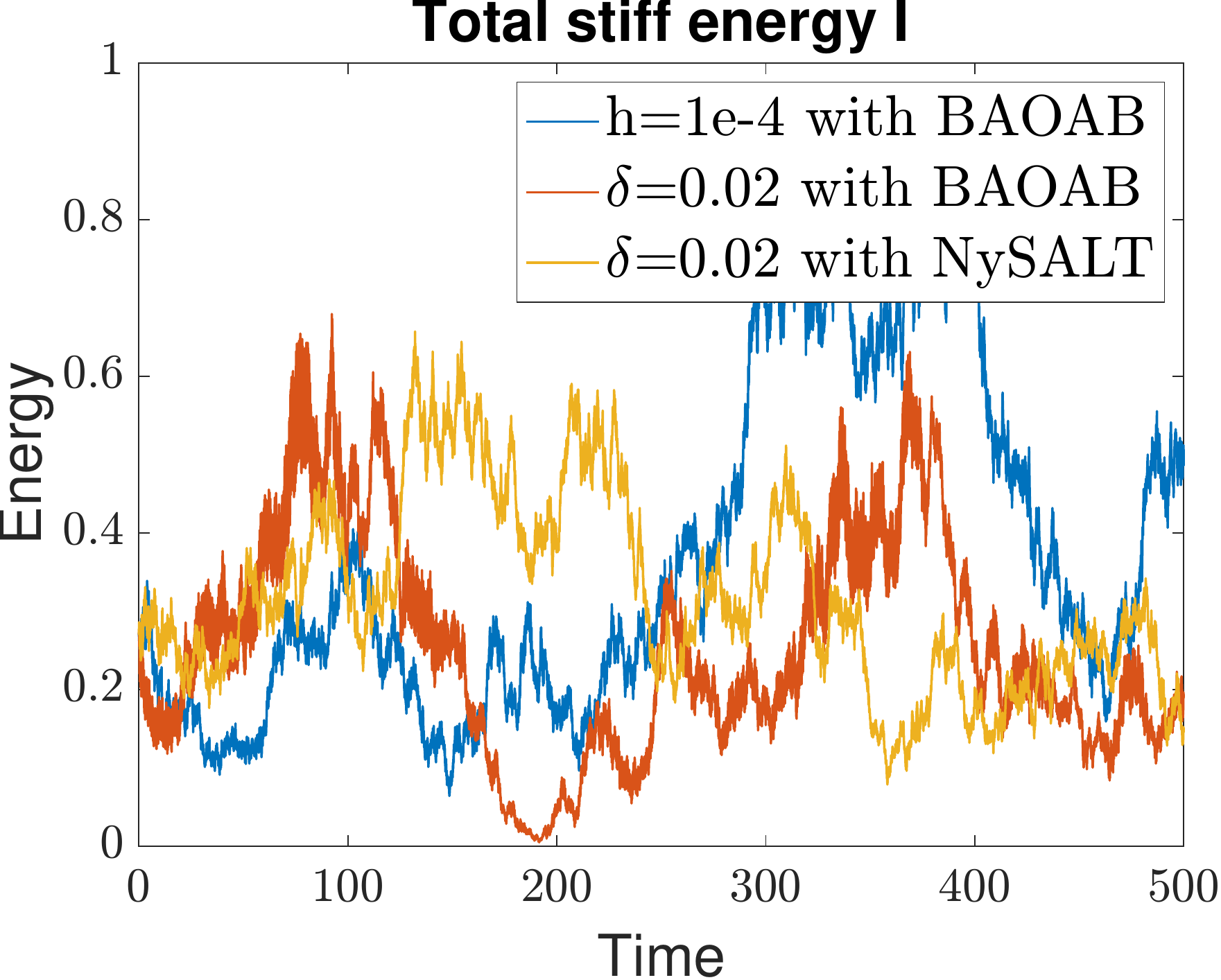}
    \includegraphics[width =0.44\textwidth]{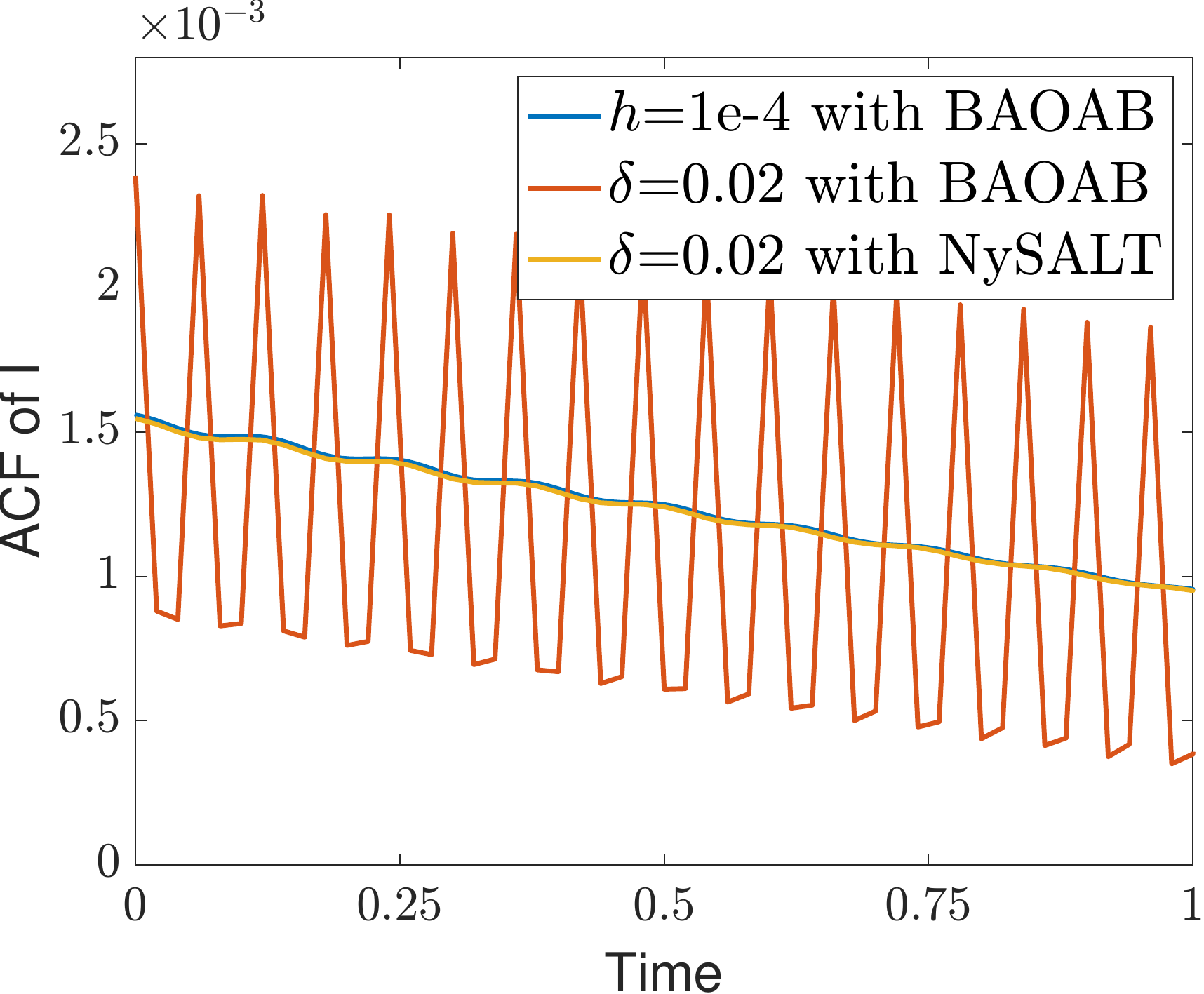}
    \caption{Comparison between the NySALT scheme with BAOAB scheme when the coarse step size is $\delta=0.02$, with the reference being the BAOAB scheme  with a fine step size $h=1\mathrm{e}{-4}$.   \textbf{Left:} Trajectories of total stiff energy $I$ by both schemes. \textbf{Right:} Time auto-covariance function (ACF) \eqref{eq:ACF} of total stiff energy $I$ by both schemes.}
    \label{Fig:SFPU_gap200}
\end{figure}

\section{A flow map approximation framework for learning integrators}\label{sec:framework}
Many classical numerical integrators are derived from (It\^o-) Taylor expansion for small time-stepping. Thus, they are universal and accurate when the time step size is small. However, they are not designed for integration with large time-stepping.  
 
We introduce a flow map approximation framework to learn numerical integrators that are adaptive to large time step size from data.
The fundamental idea is to approximate the discrete-time flow map by a function with parameters inferred from data. This approach takes advantage of the information from data, which consists of multiple trajectories generated by an accurate numerical integrator. It consists of three steps: data generation, parametric form derivation, and parameter estimation.

 The framework applies to both deterministic and stochastic dynamical systems, in which we treat the stochastic forcing as an input. In the following, we first introduce this approach, then we analyze the convergence of the parametric estimator and the error bounds of the learnt integrator.

\subsection{Flow map approximation for Hamiltonian systems}\label{sec:flowDet}
Let $\mX_t=\bcm \bm{q}_t \\ \bm{p}_t\ecm$ denote the state of a Hamiltonian system that satisfies
    \[
    \frac{d\mX_t}{dt} = b(\mX_t) 
    =\mbf{J}\nabla_{\mX}H,
    \]
where $\mbf{J} = \left(\begin{array}{cc}\mbf{0}&-\mbf{I}\\
    \mbf{I} & \mbf{0}\end{array}
    \right)$. 
The exact discrete-time flow map of $\mX_t$ on coarse  grids $\{t_i= i\delta \}_{i=0}^{N_t}$ satisfies
\[
\mX_{t_{i+1}}-\mX_{t_i}=
\deltat \mathcal{F}(\mX_{t_{i}}, \delta),
\]
where $\mathcal{F}$ is a function preserving the symplectic structure  (i.e., the phase-space volume of a closed surface is preserved). Since the coarse step $\delta$ is relatively large, a classical numerical integrator becomes inaccurate (see Figure \ref{Fig:FPU_linear_stable}). 

To obtain a numerical integrator with coarse step size $\delta$, 
we infer from multiple-trajectory data a symplectic function $F_\theta(\mX_{t_{i}}, \delta)$ that approximates the flow map $\mathcal{F}(\mX_{t_{i}},\delta)$: 
\[
\mX_{t_{i+1}}-\mX_{t_i} \approx
\deltat F_\theta(\mX_{t_{i}},\delta). 
\]
Here $\{F_\theta(\mX_{t_{i}},\delta), \theta\in \Theta\}$ is a family of parametric functions, whose parametric form comes from classical numerical integrators (see discussions below). The inference procedure consists of three steps: data generation, parametric form derivation, and parameter estimation. 

\paragraph{Data generation.} We generate data consisting of multiple trajectories with random initial conditions, utilizing an accurate classical numerical scheme with a fine step size $h$, which is much smaller than $\delta$, i.e., $\delta= \Gap\cdot h $. The initial conditions $\{\mX_{t_0}^{(m)}\}_{m=1}^M$ are sampled from a given distribution $\mu$ on  $\R^{2d}$, so that the trajectories explore the flow map sufficiently. 
 Then, we down-sample these trajectories to obtain training data on coarse grids $\{t_i= i\delta \}_{i=0}^{N_t}$, which we denote as 
\[ 
\textbf{Data: } \{\mX_{t_i}^{(m)},i=0,\ldots,N_t\}_{m=1}^M =\left\{\bcm \mbf{q}^{(m)}_{t_i}\\ \mbf{p}^{(m)}_{t_i}\ecm  ,i=0,\ldots,N_t\right\}_{m=1}^{M}.
\]

\paragraph{Parametric form from the Nystr\"{o}m family.}
The major difficulty in our inference-based approach is the derivation of the parametric symplectic maps. The symplectic structure is crucial for the flow maps of Hamiltonian systems. We propose to utilize the family of classical numerical integrators, particularly those come with parameters. As discussed in the introduction, there is a rich class of structure-preserving numerical integrators that are dedicated to long-time simulation of Hamiltonian systems. For simplicity as well as flexibility, we consider the family of the explicit 2-step Nyst{\"{o}}m method for the parametric function  $F_{\theta}(\mX_{t_{i}},\delta)$.

Specifically, for the system \eqref{eq:Hamil_pq}, we consider the parametric function $F_{\theta}$ from the $2$-step Nyst{\"{o}}m method in \eqref{eq:nystrom2}:  
\begin{equation}\label{flow_map_deter}
\begin{aligned}
 F_{\theta}(\mX_{t_{i}}, \delta) =& \frac{1}{\delta}\left(S_{b_1, \beta_1}^\delta(\mX_{t_{i}}\right) - \mX_{t_{i}})
  =\left(
  \begin{array}{c}
{\mbf{p}}_{t_{i}}
+\deltat\left(\beta_1 \ell_1+\beta_2\ell_2\right)\\
 \left(b_1 \ell_1+b_2\ell_2\right)
  \end{array}
  \right),
  \end{aligned}
\end{equation}
where the terms $\ell_1$ and $\ell_2$ are defined as
\begin{equation}\label{eq:l1l2}
\begin{aligned}
\ell_1= 
g(\mbf{q}_{t_{i}}+ c_1 \deltat \mbf{p}_{t_{i}}) \quad \text{and} \quad
\ell_2=
g( \mbf{q}_{t_{i}}+ c_2 \deltat \mbf{p}_{t_{i}} +
\deltat^2 a_{21}\ell_1),
\end{aligned}
\end{equation}
 with parameters $\{\beta_1,\, \beta_2,\, b_1,\, b_2,\, c_1,\,c_2,\,a_{21} \}$ satisfying \eqref{Nystrom_constraint_summary}. The free parameters to be estimated from data are 
 \begin{equation}\label{eq:thetaRange}
 \theta = (b_1,\beta_1) \in \Theta =(0,1)\times [0,\frac{1}{2}].
 \end{equation}

\paragraph{Parameter estimation.}
We estimate $\theta$ by minimizing the 1-step prediction error:   
\begin{equation}\label{opt_flow_map_Weightcost}
\theta_M^* = \argmin{\theta\in \Theta} \mathcal{E}_M(\theta), 
\end{equation}
where the loss function $\mathcal{E}_M(\theta)$ is the 1-step prediction error and is computed from data: 
\begin{align}
    \mathcal{E}_M(\theta)  & = \frac{1}{MN_t} \sum_{m=1}^M \sum_{i=0}^{N_t-1} \left\|F_\theta(\bm{X}_{t_{i}}^{(m)},\delta) - \mathcal{F}(\bm{X}_{t_{i}}^{(m)},\delta) \right\|_{\Sigma^{-1}}^2 \notag\\
    & = \frac{1}{MN_t} \sum_{m=1}^M \sum_{i=0}^{N_t-1}  \left\|F_\theta(\bm{X}_{t_{i}}^{(m)},\delta) -(\mX^{(m)}_{t_{i+1}}-\mX^{(m)}_{t_i})/\deltat \right\|_{\Sigma^{-1}}^2, 
    \label{eq:loss_det}
\end{align}
with $\|\bm{Y}\|_{\Sigma^{-1}}^2$ as the notation of trace norm of $\bm{Y}^T \Sigma^{-1} \bm{Y}$, i.e,  $\|\bm{Y}\|_{\Sigma^{-1}}^2 = \|\bm{Y}^T \Sigma^{-1} \bm{Y}\|_*=\text{Tr}(\bm{Y}^T \Sigma^{-1} \bm{Y})$. Here $\Sigma$ is a diagonal weight matrix aiming to normalize the contributions of the entries. We set $\Sigma$ to be a diagonal matrix with diagonal entries being the  the mean of entrywise square of $(\mX^{(m)}_{t_{i+1}}-\mX^{(m)}_{t_i})/\delta$.  

Notice that the optimization problem is nonlinear because of the nonlinear function $g$ in \eqref{eq:Hamil_pq}. 
Since the parameter is in a 2D rectangle and  the loss function is smooth, we solved it by constrained nonlinear optimization with the interior point algorithm.

As to be shown in Section \ref{sec:conv}, the estimator $\theta_M^*$ converges almost surely regarding $M$ under suitable conditions on the uniqueness of the minimizer. We select a stabilized estimator (when the sample size is sufficiently large) as $\theta^*$ for our NySALT scheme 
\begin{equation}\label{inferr_scheme_deter}
    \mX_{t_{i+1}} - \mX_{t_{i}} = \deltat F_{\theta^*}(\mX_{t_{i}},\delta).
\end{equation}

 \subsection{Flow map approximation for Langevin systems}    \label{sec:flowSto}
The governing equation \eqref{eq:Hamil_pqstoc}  to Langevin dynamics is written as
\begin{equation}
    d\mX_t= b(\mX_t)dt+\sigma(\mX_{t_i})  d\mW_t.
\label{eq:generalSDE}
\end{equation}
In integral form, we can write exact solution $\mX_{t}$ on each coarse grid $\{t_i\}$ as  
\begin{equation}\label{eq:flowmap}
 \mX_{t_{i+1}}  - \mX_{t_{i}} =\int_{t_i}^{t_{i+1}} b(\mX_s)ds+ \sigma(\mX_{t_i}) (\mW_{t_{i+1}} -\mW_{t_i}) = \deltat \mcF(\mX_{t_i}, \mW_{[t_i,t_{i+1}]},\delta).
\end{equation}
Here the discrete-time flow map $\mcF(\mX_{t_i}, \mW_{[t_i,t_{i+1}]}, \delta)$ is an infinite-dimensional functional that depends on the path of the Brownian motion  $\mW_{[t_i,t_{i+1}]}$. In general,  a numerical scheme approximates the discrete-time flow map by a function depending on $\mX_{t_i}$ and a low-dimensional approximation of the Brownian path $\mW_{[t_i,t_{i+1}]}$ (either in distribution in the weak sense or trajectory-wisely in the strong sense).  For example, the Euler-Maruyama scheme gives the function $ F(\mX_{t_i},  \xi_i,\delta) = b(\mX_{t_i}) + \sigma(\mX_{t_i}) \xi_i/\delta$ with $\xi_i =\mW_{t_{i+1}}- \mW_{t_i}\sim \mathcal{N}(0,\delta)$. 
Due to their reliance on the Ito-Taylor expansion, these classical schemes require a small time step for accuracy.

In order to allow a large time-stepping $\deltat$, from data we infer a parametric function $ F_\theta(\mX_{t_i},  \xi_i,\deltat)$ to approximate the flow map $\mcF(\mX_{t_i}, \mW_{[t_i,t_{i+1}]}, \delta)$, where $\xi_i$ depends on the path $ \mW_{[t_i,t_{i+1})}^{(m)}$.   Similar to the deterministic case, the inference consists of three steps: data generation, parametric form derivation, and parameter estimation.

\paragraph{Data generation.} 
The data consists of both the process $\mX_t$ and the stochastic force $\xi_t$. 
\[ 
\textbf{Data: } \{\mX_{t_i}^{(m)},\xi_{t_i}^{(m)}, i=0,\ldots,N_t\}_{m=1}^M =\left\{\bcm \mbf{q}^{(m)}_{t_i}\\\mbf{p}^{(m)}_{t_i}\ecm,\xi_{t_i}^{(m)} ,i=0,\ldots,N_t\right\}_{m=1}^{M}.
\]
The initial conditions $\{\mX_{t_0}^{(m)}\}_{m=1}^M$ are sampled from a distribution $\mu$ on $\R^{2d}$, so that the short trajectories can explore the flow map sufficiently. 
 Suppose that the system is resolved accurately by an integrator with a fine time step size $h$. Then similar to deterministic systems, a data trajectory $\mX_{t_i}$ is obtained by down-sampling the fine solution with coarse grid $\{t_i = i\delta \}_{i=1}^{N_t}$.

However, the stochastic force cannot be down-sampled directly since one has to follow the desired distribution in \eqref{eq:SO_def}. 
Here we use the one-step increment of OU process to approximate $\xi_{t_i}$ which takes into account the friction and the noise. Consider the OU process $dY_t = -\gamma Y_t dt + \sigma d\mW_t$, the solution of this OU process with coarse step $\delta$ is expressed by using noise with fine time step $h$, 

\begin{align*}
Y_{\delta} &= e^{-\gamma \delta} Y_0+\sigma \int_0^\delta e^{-\gamma (\delta-s)}d\mW_s 
= e^{-\gamma \delta } Y_0+\sigma \sum_{j=1}^\Gap \int_{jh-h}^{jh} e^{-\gamma (\delta-s)}d\mW_s \\
& \sim  e^{-\gamma \delta } Y_0 + \sigma \sum_{j=1}^\Gap \sqrt{\frac{1}{2\gamma }(1-\exp^{-2\gamma h}) } e^{-\gamma (\Gap-j)h}(\mW_{jh} - \mW_{jh-h})/\sqrt{h}.
\end{align*}
where in the last step we used the fact that $\int_a^b e^{-{\gamma (t-s)}}d\mW_s \sim \mathcal{N}(0, \frac{1}{2\gamma} (e^{-2\gamma a} -e^{-2\gamma b}) )$. 
Then the one-step increment at time instants $t_i$ can be approximated by
\begin{align}\label{eq:Xi_data}
\xi_{t_{i}} = \sigma \sqrt{\frac{1}{2\gamma}(1-\exp^{-2\gamma {h}}) } \sum_{j=1}^{\Gap} e^{-\gamma j h}   R_{i,j}. 
\end{align}
where $R_{i,j} =(\mW_{((i-1)\Gap +j)h} -\mW_{((i-1)\Gap+j-1)h})/ \sqrt{h}  $ is the scaled increment of the Brownian motion. Consequently, one can show that $\xi_{t_i}\sim \mathcal{N}\left(0,  {\frac{\sigma^2}{2\gamma}(1-\exp^{-2\gamma \delta}) } \right)$.

\paragraph{Parametric form from the the Nystr\"{o}m family.} We approximate the flow map by the parametric function $F_{\theta}(\mX_{t_{i}},\xi_{t_{i}}, \deltat )$ in the stochastic symplectic Nystr\"{o}m scheme introduced in \eqref{eq:SO_def}. 
Note that it consists of the symplectic integrator $S^{h}_{b_1,\beta_1}$ and the Ornstein-Uhlenbeck integrator, 
\begin{equation}\label{flow_map_sto}
\begin{aligned}
 F_{\theta}(\mX_{t_{i}}, \xi_{t_{i}}, \delta) =& 
 \left(
  \begin{array}{c}
{\mbf{p}}_{t_{i}}
+\deltat\left(\beta_1 \ell_1+\beta_2\ell_2\right)\\
 \left(\frac{\exp(-\gamma h) -1}{\delta}\right)\mbf{p}_{t_{i}}+\exp(-\gamma h)  \left(b_1 \ell_1+b_2\ell_2\right) + \frac{\xi_{t_i}}{\delta}
  \end{array}
  \right),
  \end{aligned}
\end{equation}
where $\ell_1$ and $\ell_2$ are defined in \eqref{eq:l1l2}. Thus, it has the same parametric form as the symplectic integrator $S^{h}_{b_1,\beta_1}$ as the deterministic case, and the range for the parameter keeps the same as in \eqref{eq:thetaRange}.

\paragraph{Parameter estimation.}
 The parameter $\theta\in \Theta$ is estimated by minimizing the 1-step prediction error 
\begin{align}\label{opt_flow_map_Weightcost_sto}
\theta_M^* &= \argmin{\theta\in \Theta} \mathcal{E}_M(\theta), \text{ with } \\ \label{eq:loss_sto}
\mathcal{E}_M(\theta) &= \frac{1}{MN_t} \sum_{m=1}^M \sum_{i=0}^{N_t-1} \left\| \left(\mX_{t_i}^{(m)}+\delta F_\theta(\mX_{t_i}^{(m)}, \xi_i^{(m)},\deltat )\right) - \mX_{t_{i+1}}^{(m)} 
\right\|_{\Sigma^{-1}}^2, 
\end{align}
where $\{\mX_{t_{i}}^{(m)}, \xi_i^{(m)}\}_{m=1}^M$ are down-sampled fine scale data consisting of $M$ trajectories of the state and the coarsened increments of stochastic force. $\Sigma$ is the diagonal weight matrix to normalize the contribution of $\mbf{p}$ and $\mbf{q}$. We set $\Sigma$ to be a diagonal matrix with diagonal entries being the mean of the square of $(\mX^{(m)}_{t_{i+1}}-\mX^{(m)}_{t_i})$.  We note that the loss function is not the log-likelihood of the data, which is not available because the transition density of the stochastic symplectic Nystr\"{o}m scheme scheme is nonlinear and non-Gaussian without an explicit form.

With the gradient of the 1-step prediction error explicitly calculated in Appendix \ref{appendix-A}, we solve this constrained nonlinear optimization with the interior point algorithm.  
Since the estimator $\theta_M$ converges almost surely as $M$ increases (see Section \ref{sec:conv}), we select a stabilized estimator as $\theta^*$ for our inferred scheme, 
\begin{equation}\label{inferr_scheme_sto}
    \mX_{t_{i+1}} - \mX_{t_{i}} = \deltat F_{\theta^*}(\mX_{t_{i}},\xi_{t_i}, \deltat),
\end{equation}
where $\xi_{t_i}$ is sampled from $\mathcal{N}\left(0,  {\frac{\sigma^2}{2\gamma}(1-\exp^{-2\gamma \delta}) } \right) $, 
the same distribution as the coarsened increments of the OU process.

\subsection{Convergence of the parameter estimator}\label{sec:conv}
We show that the parameter estimator converges as the number of independent data trajectories increases, under suitable conditions on the loss function. These conditions require the parametric function $F_\theta$ to be continuously differentiable in $\theta$, along with integrability conditions that generally hold true for Hamiltonian systems and symplectic integrators.

For simplicity of notation, we denote the loss for each data trajectory by 
\begin{equation}\label{eq:loss1traj}
L(\theta) =  
\begin{cases}
\frac{1}{N_t}\sum_{i=0}^{N_t-1}  \left\|F_\theta(\bm{X}_{t_{i}},\delta) - \mathcal{F}(\bm{X}_{t_{i}},\delta) \right\|_{\Sigma^{-1}}^2, \quad \text{for deterministic \eqref{eq:loss_det}}, \\
\frac{1}{N_t}  \sum_{i=0}^{N_t-1} \left\| \delta \left(F_\theta(\mX_{t_i}, \xi_{t_i},\deltat ) -  \mcF(\mX_{t_i}, \mW_{[t_i, t_{i+1}]},\deltat)\right)\right\|_{\Sigma^{-1}}^2, \quad \text{for stochastic \eqref{eq:loss_sto}}. 
\end{cases}
\end{equation}
Then, $\calE_M(\theta) =  \frac{1}{M}\sum_{m=1}^{M} L^{(m)}(\theta) $, where $L^{(m)}(\theta) $ is the loss of the $m$-th data trajectory $\mX^{(m)}$. 

Hereafter, we denote $\mathbb{P}$ the probability measure that characterizes the randomness coming from initial conditions and the stochastic driving force. We denote $\E$ the corresponding expectation.  

\begin{assumption}\label{assum}
 We make the following assumptions:
\begin{itemize}
\item[{\rm (a)}] $\E[L(\theta)] \in C^2(\Theta)$, $\E[|\nabla L(\theta)|^2]< \infty$ and $\E[|\nabla^2 L(\theta)|]<\infty$ for any $\theta\in \Theta^o$,  the interior of $\Theta$.
\item[{\rm (b)}] $\theta^*\in \Theta^o$ is the unique minimizer of $\E [L(\theta)]$ in $\Theta$. 
\item[{\rm (c)}] There exists $C>0$, $p\geq 1$ and $q>1$ such that $\E[|L(\theta_1)- L(\theta_2)|^{2p}] \leq C|\theta_1-\theta_2|^q$ for any $\theta_1,\theta_2\in\Theta$. 
\end{itemize}
\end{assumption}

\begin{theorem}\label{thm_convEst} Under Assumption {\rm \ref{assum}}, the estimator $\theta_M$ in either \eqref{opt_flow_map_Weightcost} or \eqref{opt_flow_map_Weightcost_sto} converges to $\theta^*$ in probability and  $\sqrt{M} (\theta_{M}- \theta^*)$ is asymptotically normal as $M\to \infty$. 
\end{theorem}
\begin{proof} First, we show that  $\theta_M$ converges to $\theta^*$ in probability, i.e., for any $\nu>0$, $\lim_{M\to \infty} \prob{|\theta_M-\theta^*|>\nu} =0$, where $\prob{A}$ is the probability of an event $A$.  

Note that for any $(\theta_1,\ldots,\theta_k)\subset \Theta$, as $M\to \infty$, we have the convergence in probability of the vectors 
\[(\calE_M(\theta_1,\ldots,\calE_M(\theta_k)) \to ((\E[L(\theta_1)],\ldots,\E[L(\theta_k)) ])\]
by the law of large numbers. Together with Assumption \ref{assum} (c), they imply that the measure induced by $\calE_M(\cdot)$ on $(C(\overline \Theta),\mathcal{B})$, the space of continuous functions on $\overline \Theta$ with uniform metric and with $\mathcal{B}$ being the $\sigma$-algebra of Borel subsets, converges to the measure induced by $\E[L(\cdot)]$ (see \cite[Lemma 1.33, page 61]{Kut04} and \cite[Theorem 13.2]{billingsley2013convergence}). Then, any continuous functional  of the process $\calE_M(\cdot)$ converges in probability as $M\to \infty$. In particular, we have for any $\nu>0$
\[
\prob{\sup_{|\theta-\theta^*|>\nu} \calE_M>  \sup_{|\theta-\theta^*|<\nu} \calE_M} \to \prob{\sup_{|\theta-\theta^*|>\nu} \E[L(\theta)]>  \sup_{|\theta-\theta^*|<\nu} \E[L(\theta)]} =0,
\]
where the equality follows from Assumption \ref{assum} (b). Meanwhile, note that by the definition of $\theta_M$ in \eqref{eq:loss_sto}, we have
\[ 
\prob{|\theta_M-\theta^*|>\nu} = \prob{\sup_{|\theta-\theta^*|>\nu} \calE_M>  \sup_{|\theta-\theta^*|<\nu} \calE_M}. 
\]
Combining the above two equations, we obtain the convergence in probability of $\theta_M$ to $\theta^*$. 

Next, we show that $\sqrt{M} (\theta_{M}- \theta^*)$ is asymptotically normal. Since $\theta_M$ is a minimizer of $\calE_M$, we have
\[
0= \nabla \calE_M(\theta_M) = \nabla \calE_M(\theta^*) + \nabla^2\calE_M(\widetilde \theta_M) (\theta_M-\theta^*),  
\]
where $\widetilde \theta_M = \theta^*+ s(\theta_M-\theta^*)$ for some $s\in[0,1]$. 

Note first that $ \nabla^2\calE_M(\widetilde \theta_M)$ converges in probability to $\E[\nabla^2 L(\theta^*)] $. It follows by the law of large numbers, Assumption {\rm \ref{assum}}(a), and the consistency of $\theta_M$, which implies that $\widetilde \theta_M$ converges to $\theta^*$. Thus, the inverse of the matrix $\nabla^2\calE_M(\widetilde \theta_M)$ exists when $M$ is large, because $\E[\nabla^2 L(\theta^*)] $ is strictly positive definite. 
 Thus, 
\[
\theta_M-\theta^* = \nabla^2\calE_M(\widetilde \theta_M)^{-1} \nabla \calE_M(\theta^*). 
\]

Note also that $\nabla \calE_M(\theta^*) = \frac{1}{M}\sum_{m=1}^M\nabla L^{(m)}(\theta^*)$ and $\E[\nabla \calE_M(\theta^*)] = \E[\nabla L(\theta^*)] =0$ because $\theta^*$ is the unique minimizer by Assumption {\rm \ref{assum}}(b).
Thus, by the central limit theorem, we have the convergence in distribution:
\[ \sqrt{M} \nabla \calE_M(\theta^*) \to \mathcal{N}(0, \Sigma_L), 
\]
 where $\Sigma_L $ is the covariance of $\nabla L(\theta^*)$. 
 
 Combining the above, we obtain the asymptotic normality of $\sqrt{M}(\theta_M-\theta^* ) $.  
\end{proof}

\subsection{Statistical error at arbitrary time } 
\paragraph{Hamiltonian systems}
Since the learned integrator is a symplectic partitioned Runge-Kutta method, by \cite[Theorem IX.3.3]{Hairer2006}, any trajectory it generates exactly corresponds to time-discretized stroboscopic samples of a continuous solution of some fixed modified Hamiltonian $\tilde{H}$, at least formally. If the learned integrator has local truncation error of order $p+1$, then we have $\tilde{H}=H+\mathcal{O}(h^p)$ at least formally speaking. Assume the original Hamiltonian system is integrable, analytic, and the initial condition corresponds to frequency vector that are in a sufficiently small neighborhood of some Diophantine frequency vector, then by \cite[Theorem X.3.1]{Hairer2006}, the learned integrator has a linearly growing long time error bound. More precisely, 
\[
    \|\mX_{t_i}-\mX(t_i)\| \leq C h^{p+1} i,
\]
for at least $i\leq \hat{C} h^{-p-1}$, where $\mX_{t_i}$ is the numerical solution given by the learned integrator and $\mX(t_i)$ is the exact solution of the original Hamiltonian. Moreover, for any action variable $\mbf{I}(\mX)$, it is nearly conserved over long time, i.e.,
\[
    |\mbf{I}(\mX_{t_i})-\mbf{I}(\mX_{t_0})| \leq C h^p.
\]

In Section~\ref{sec:FPU_def}, we test the numerical accuracy with respect to step size $\delta$ over different time periods, that is $T_{\text{test}}=0.5$ and $T_{\text{test}}=100$. Note it is difficult to quantitatively put these values in the context of the above discussion, because the validity timespan of $i\leq \hat{C} h^{-p-1}$ may not be the longest possible (see e.g., \cite{benettin1994hamiltonian} for possible exponential results), and constants such as $\hat{C}$ may not be explicit. 

\paragraph{Langevin dynamics}
If the Langevin dynamics \eqref{eq:generalSDE} is contractive in the sense that
  there exists a constant matrix $A$, and constants $t_0>0$ \& $\beta>0$, s.t. for any two solutions $\mX(t)$,  $\mY(t)$ driven by the same stochastic forcing (i.e. synchronous coupling),
\[				    \left(\mathbb{E} \|A\left(\mX(t) - \mY(t)\right)\|^2\right)^\frac{1}{2} \le \left(\mathbb{E}\|A\left(\mX(0) - \mY(0) \right)\|^2\right)^\frac{1}{2} e^{-\beta t}, \quad \forall 0\leq t < t_0,
\]
then the framework of mean-square analysis for sampling described in \cite{li2022sqrt} can help obtain a bound of the statistical error of the learned integrator at any $t\in[0, \, t_0)$ (which also means for any number of steps $k$ as $t=kh$). In particular, the kinetic Langevin equation \eqref{eq:Hamil_pqstoc} is known to be contractive when $\gamma$ is large enough (e.g., \cite{dalalyan2020sampling}) and when the potential $V$ is strongly-convex and admitting a Lipschitz gradient. 

In addition, because our NySALT scheme is a Lie-Trotter composition of a consistent Hamiltonian integrator (due to being Nystr\"om) and an exact OU process, the local weak error is at least of order 1 and the local strong error is at least of order $1/2$, 
(see e.g., \cite{milstein2004stochastic}). Therefore, conditions of \cite[Theorem 3.3]{li2022sqrt} are satisfied with $p_1=1$ and $p_2=1/2$. Consequently, \cite[Theorem 3.4]{li2022sqrt} gives
\[
    W_2(\text{Law}(\mX_{t_k}), \mu)
	\le e^{- \beta kh}  W_2(\text{Law}(\mX_{t_0}), \mu) + C h^{1/2}, \quad \forall 0 < h \le h_1
\]
for some explicitly obtainable constant $C$ and $h_1$, where $\mu$ is the ergodic measure associated with the original SDE (i.e., \eqref{eq:Hamil_pqstoc}),
$\mX_{t_k}$ is the numerical solution produced by the NySALT, and $W_2(\cdot)$ is the 2-Wasserstein distance 
$W_2(\mu_1, \mu_2):=\left(\text{inf}_{(\mX,\mY)\sim \Pi(\mu_1,\mu_2)}\mathbb{E}\|\mX-\mY\|^2\right)^{1/2}$. 

\section{Optimal parameters for linear systems}\label{sec:linearAnalysis}
To demonstrate the discrete-time flow map, we first consider  linear systems and show the estimation of optimal parameters. For simplicity of notation, we consider only 1D systems and the extension to higher dimensional systems is straightforward.

\subsection{Linear Hamiltonian systems} \label{sec:lin-Ham}
We first consider a one-dimensional linear Hamiltonian system 
\begin{equation}\label{eq:Linear}
\begin{aligned}
\begin{cases}
\dot{\mbf{q}}=\mbf{p},\\
\dot{\mbf{p}}=-\Omega \mbf{q} . 
\end{cases}
\end{aligned}
\end{equation}
with $\Omega=\omega^2$ and $\mbf{q}=\mbf{q}(t):[0,\, T]\rightarrow \mathbb{R}$.

\begin{proposition} \label{prop:linear_para}
Let $\{\mathbf{X}_{t_i}, i=1,\ldots,N_t\}_{m=1}^M$ be $M$ independent solution trajectories to \eqref{eq:Linear} with time step $\delta = t_{i+1}-t_i$ for all $i$. Then, the loss function \eqref{eq:loss_det} becomes
\begin{equation}\label{eq:loss_linearDet}
\mcE_M(b_1,\beta_1) = \frac{1}{MN_t\delta^2}\sum_{m=1}^M \sum_{i=1}^{N_t} \| ( e^{A \delta} - B_{b_1,\beta_1}^\delta ) \mathbf{X}^{(m)}_{t_i}\|_{\Sigma^{-1}}^2,  
\end{equation}
where $\Sigma = \begin{pmatrix}1 & 0\\ 0& \Omega^{2} \end{pmatrix}$ is  the mean of square of $\Delta \mathbf{X}/\delta$, $A= \bcm
0& 1\\ -\Omega, & 0 
\ecm$ and 
\begin{align}
 B_{b_1,\beta_1}^\delta = \bcm
1 - \frac{1}{2}\delta^2\Omega + \delta^4\Omega^2\beta_2a_{21}, & \delta -\delta^3 \Omega(\beta_1 c_1+ \beta_2 c_2) +\delta^5\Omega^2\beta_2a_{21} c_1 \\ 
 -\delta\Omega + \delta^3 \Omega^2 b_2 a_{21}, & 1 - \frac{1}{2}\delta^2\Omega+ \delta^4\Omega^2 b_2a_{21}c_1\ecm
\end{align}
with $a_{21} = (\beta_1 -\frac{b_1}{b_2} \beta_2), c_1=1-\frac{\beta_1}{b_1}, c_2= 1-\frac{\beta_2}{b_2}, b_2 = 1- b_1$ and $\beta_2=1-\beta_1$.   
In particular, when the  $\mathrm{span}\{\mathbf{X}_{t_i}^{(m)}, i=1,\ldots,N_t,m=1,\ldots,M\} = \R^2$, then the cost function has the same minimizer as $\|( e^{A \delta} - B_{b_1,\beta_1}^\delta )\|_{\Sigma^{-1}}^2$ does.   
\end{proposition}

\begin{proof}
Denote the discrete Nystr\"{o}m solution by $(\mbf{q},\, \mbf{p})$.  
At $t_{i+1}$, the Nystr\"{o}m method gives
\[
\begin{aligned}
\ell_1=-\Omega \big(\mbf{q}_{t_i}+c_1\delta \mbf{p}_{t_i}\big), \quad &\text{ and } \quad \ell_2= -\Omega\big( 
\mbf{q}_{t_i}+c_2\delta \mbf{p}_{t_i}+\delta^2 a_{21}\ell_1\big),\\
\mbf{q}_{t_{i+1}}=
\mbf{q}_{t_i}+\delta \mbf{p}_{t_i}+\delta^2\big(\beta_1\ell_1+\beta_2\ell_2\big), \quad &\text{ and } \quad  \mbf{p}_{t_{i+1}}=
\mbf{p}_{t_i}+\delta\big(b_1\ell_1+b_2\ell_2\big),
\end{aligned}
\]
where the parameters $\big\{\{\beta_k\}_{k=1}^2,\,\{b_k\}_{k=1}^2,\,\{c_k\}_{k=1}^2,\, a_{21}\big\}$ satisfy the constraints \eqref{Nystrom_constraint_summary}.
Simplifying the expressions of $\mbf{q}_{t_{i+1}}$ and $\mbf{p}_{t_{i+1}}$ by the constraints, we get
\begin{equation}\label{Nystrom_sol_1step}
    \begin{aligned}
 \mbf{q}_{t_{i+1}}=&  \mbf{q}_{t_{i}}
 +\delta \mbf{p}_{t_{i}}-\frac{1}{2}\delta^2\Omega \mbf{q}_{t_{i}}
 -\delta^3 \Omega(\beta_1c_1+\beta_2c_2)\mbf{p}_{t_{i}} +\delta^4\Omega^2 \beta_2a_{21}\mbf{q}_{t_{i}}
 +\delta^5\Omega^2\beta_2a_{21}c_1\mbf{p}_{t_{i}},\\
  \mbf{p}_{t_{i+1}}=&\mbf{p}_{t_{i}}-\delta \Omega\mbf{q}_{t_{i}}-\frac{1}{2}\delta^2\Omega \mbf{p}_{t_{i}} +\delta^3 \Omega^2 b_2 a_{21}\mbf{q}_{t_{i}}+\delta^4\Omega^2 b_2a_{21}c_1\mbf{p}_{t_{i}}.
    \end{aligned}
\end{equation}
With the notation $\mathbf{X}_{t_i} = (\mbf{q}_{t_i},\mbf{p}_{t_i} )$, we can write the above Nystr\"om algorithm as
\begin{equation}\label{eq:linearNyst}
     \Xb^N_{t_{i+1}} =  B_{b_1,\beta_1}^\delta  \Xb_{t_{i}}.
\end{equation}
Comparing with the exact solution:  
$\Xb_{t_{i+1}} = e^{A\delta}\Xb_{t_{i}}$. 
We can write the 1-step prediction error as 
\[
 \Xb_{t_{i+1}} - \Xb^N_{t_{i+1}} = ( e^{A\delta} - B_{b_1,\beta_1}^\delta ) \Xb_{t_i}.
\]
Then, with the data, we obtain the cost function \eqref{eq:loss_linearDet}.
\end{proof}

\begin{remark}[Optimal parameter for the linear Hamiltonian system] \label{rmk:optEst_linear}
The minimizers of $\mcE^\delta(b_1,\beta_1)$ in \eqref{eq:loss_linearDet} are close to $b_1^* =0.5$ and $\beta_1^*\approx 0.40$, as shown in Figure {\rm \ref{Fig:para_vs_omega}}. They appear to be independent of $\delta$ because the loss function depends on $\delta$ are in high-orders, which can be seen from a Taylor expansion of $\| B_{b_1,\beta_1}^\delta - e^{A\delta}\|^2_{\Sigma^{-1}}$ up to third order  as follows. 
 Note that $ e^{A\delta} =     I_2 + A\delta +A^2\frac{\delta^2}{2}    + \bcm    0 & -\Omega \\ \Omega^2 & 0     \ecm\frac{\delta^3}{6} + O(\delta^4) $ and 
\begin{align*}
   B_{b_1,\beta_1}^\delta = I_2 + A\delta +A^2\frac{\delta^2}{2}  
    + \bcm
    0 & -6\left(\frac{1}{2} -\frac{\beta_1^2}{b_1}-\frac{\beta_2^2}{b_2}\right)\Omega \\ 6b_2\left(\beta_1-\frac{b_1}{b_2}\beta_2\right)\Omega^2 & 0 
    \ecm\frac{\delta^3}{6} + O(\delta^4). 
\end{align*}
Thus, the trace norm of the discrepancy matrix is 
\begin{align*}
   \| B_{b_1,\beta_1}^\delta - e^{A\delta}\|^2_{\Sigma^{-1}}& =  \left\| \bcm
    0 & \Omega-6\left(\frac{1}{2} -\frac{\beta_1^2}{b_1}-\frac{\beta_2^2}{b_2}\right)\Omega \\ -\Omega^2+6b_2\left(\beta_1-\frac{b_1}{b_2}\beta_2\right)\Omega^2 & 0  \ecm \frac{\delta^3}{6} + O(\delta^4)\right\|^2_{\Sigma^{-1}} \\
    & =  \frac{\delta^6}{36}\Omega^2 \left(\left(\frac{6\beta_1^2}{b_1}+\frac{6\beta_2^2}{b_2}-2\right)^2 + (6b_2\beta_1-6b_1\beta_2-1)^2\right) + O(\delta^8).
\end{align*}
The minimum of the function $f(b_1,\beta_1) = \left(\frac{6\beta_1^2}{b_1}+\frac{6\beta_2^2}{b_2}-2\right)^2 + (6b_2\beta_1-6b_1\beta_2-1)^2$ is reached at  $b_1^* =0.5$ and $\beta_1^*\approx 0.40$. 
Note also that this estimator is independent of $\Omega$, because of the weight matrix $\Sigma$.
\end{remark}

\begin{remark}[Maximal admissible step size of linear stability] \label{rmk:linearStab}
The largest time step size of linear stability for the Nystr\"{o}m integrator \eqref{eq:linearNyst} is determined by $B^\delta_{b_1,\beta_1}$. It is the largest $\delta$ such that the real parts of the eigenvalues of $B^\delta_{b_1,\beta_1}$ are less than or equal $1$. 
For the Nystr\"{o}m integrator with optimal parameters $(b_1^*,\beta_1^*)= (0.5,0.40)$ estimated in Remark {\rm \ref{rmk:optEst_linear}}, we have $B^\delta_{b_1^*,\beta_1^*} =  \begin{pmatrix} 1-0.5z+0.03z^2 & \delta (1-0.16z+0.006z^2) \\ \frac{z}{\delta}(-1+0.15z) &1-0.5z+0.03z^2      \end{pmatrix}  $
with $z=\delta^2\Omega$. Thus, it can be verified directly that $\mathrm{det}(B^\delta_{b_1^*,\beta_1^*} )=1$, and its eigenvalues are $\lambda_{1,2} = a \pm \sqrt{a^2-1}$ with $a=1-0.5z+0.03z^2 $. Thus, to have $\mathrm{Real}(\lambda_{1,2})\leq 1$, we need $|a|\leq 1$, which implies either $0\le z\le \frac{20}{3}$ or $10\le z\le \frac{50}{3}$. Therefore, to ensure the linear stability as well as consistency, the largest time step of linear stability is $\delta^*\le \frac{\sqrt{20/3}}{\omega}$, which is $\sqrt{\frac{5}{3}}$ times the Verlet method's linear stability $\frac{2}{\omega}$. 
Therefore, the linear stability of NySALT scheme is improved.
\end{remark}

\subsection{Linear Langevin systems} \label{sec:lin-langevin}
We can estimate the optimal parameters from the analytical solutions of one-dimensional linear Langevin systems. Recall that for the governing equations
\[ 
d\mathbf{X}_t = A_\gamma\mathbf{X}_t + \sigma \begin{pmatrix}
0\\ d{\mbf{W}_t}\end{pmatrix} 
\]
with $A_\gamma = \bcm
0& 1\\ -\Omega, & -\gamma 
\ecm$, the exact solution is 
\begin{equation}\label{eq:Solu_stoc}
\mathbf{X}_{t_{i+1}} = e^{A_\gamma\delta} \mathbf{X}_{t_i} + \bm{W}_{t_i}^\delta, \quad \bm{W}_{t_i}^\delta= \sigma \int_{t_i}^{t_{i}+\delta} e^{A_\gamma(t_i+\delta - s)}\begin{pmatrix}
0\\ d\mW_s\end{pmatrix}. 
 \end{equation}
The Stochastic Symplectic Nystr\"{o}m scheme for this linear system gives,
\begin{equation}\label{eq:sto_Nys}
\mathbf{X}^N_{t+\delta} = B_{b_1,\beta_1,\gamma}^\delta   \mathbf{X}_{t} +\begin{pmatrix}
0\\ \xi_t^\delta \end{pmatrix},
\end{equation}
where $B_{b_1,\beta_1,\gamma}^\delta = \bcm
1& 0\\ 0 & e^{-\gamma \delta} 
\ecm B_{b_1,\beta_1}^\delta $ and $\xi_t^\delta = \sigma \int_t^{t+\delta} e^{-\gamma (t+\delta-s)}d\mW_s$ comes from \eqref{eq:Xi_data} that uses the increments of the Brownian motion. 
Then, the 1-step prediction error gives us the cost function 
\begin{equation}\label{eq:cost_linearSto}
\mcE_M(b_1,\beta_1) = \frac{1}{MN_t}\sum_m^M \sum_{i=1}^{N_t}  \| (e^{A_\gamma\delta} - B_{b_1,\beta_1,\gamma}^\delta ) \mathbf{X}^{(m)}_{t_i} + \xi_{t_i}^\delta - \bm{W}_{t_i}^\delta\|_{\Sigma^{-1}}^2. 
\end{equation}
\begin{remark}[Optimal parameter for the linear Langevin system] 
The minimizer of this cost function depends on $\gamma$ and the data, unlike the case of deterministic linear system. Fortunately, the noise term $\xi_{t_i}^\delta - \bm{W}_{t_i}^\delta$ is centered Gaussian and is independent of $\mathbf{X}^{(m)}_{t_i}$, thus, the minimizer is still mainly determined by the discrepancy matrix $\big(e^{A_\gamma\delta} - B_{b_1,\beta_1,\gamma}^\delta\big) $. The following computation shows that the minimizer of $\left \| B_{b_1,\beta_1,\gamma}^\delta -e^{A_\gamma \delta}\right\|_{\Sigma^{-1}}^2$ is about $b_1^*=0.5$ and $\beta_1^* \approx 0.40 -\frac{0.43\gamma}{\Omega\delta}$, when $\gamma$ is sufficiently small. Thus, the optimal parameters depend on $\gamma$ and $\delta$. 

The computation is based on the Taylor expansion of $e^{A_\gamma\delta}$ and  $B_{b_1,\beta_1,\gamma}^\delta$. 
Expand $B_{b_1,\beta_1,\gamma}^\delta$ up to the order of $\delta^3$, 
\begin{align*}
    B_{b_1,\beta_1,\gamma}^\delta &=\bcm 1 - \frac{1}{2}\delta^2\Omega  & \delta -\delta^3 \Omega(\beta_1 c_1+ \beta_2 c_2)  \\ 
 \left(-\delta\Omega + \delta^3 \Omega^2 b_2 a_{21}\right)(1-\delta\gamma+\frac{1}{2}\delta^2\gamma^2) & \left(1 - \frac{1}{2}\delta^2\Omega\right) (1-\delta\gamma+\frac{1}{2}\delta^2\gamma^2-\frac{1}{6}\delta^3\gamma^3) \ecm+O(\delta^4) \\ 
    &= I_2 + A_\gamma \delta + A_\gamma^2\frac{\delta^2}{2} + \bcm 0 & \gamma \\ \gamma\Omega & 0 \ecm\frac{\delta^2}{2} +
    \bcm 0 & -6\Omega(\beta_1c_1 + \beta_2c_2) \\ 6\Omega^2b_2a_{21}-3r^2\Omega & 3\gamma \Omega-\gamma^3\ecm  \frac{\delta^3}{6}+O(\delta^4).
\end{align*}
Similarly, expand $\exp(A_\gamma \delta)$ up to the order of $\delta^3$, 
\begin{align*}
    \exp(A_\gamma \delta) = I_2 + A_\gamma \delta + A_\gamma^2 \frac{\delta^2}{2} + \bcm \gamma\Omega & -\Omega+\gamma^2 \\ \Omega^2-\gamma^2\Omega & 2\gamma\Omega-\gamma^3 \ecm\frac{\delta^3}{6}  + O(\delta^4).
\end{align*}
Then the discrepancy matrix is approximately, 
\begin{align*}
 &  \left \| B_{b_1,\beta_1,\gamma}^\delta -\exp(A_\gamma \delta) \right\|_{\Sigma^{-1}}^2\\ 
 &\approx \left(\frac{\delta^2}{2}\right)^2\left\|\bcm 0 & \gamma \\ \gamma\Omega & 0 \ecm + 
    \frac{\delta}{3}\bcm - \gamma\Omega & \Omega(1-6(\beta_1c_1 + \beta_2c_2))-\gamma^2 \\
     \Omega^2(6b_2a_{21}-1)-2\gamma^2\Omega & \gamma\Omega\ecm \right \|_{\Sigma^{-1}}^2\\
     &= \frac{\delta^6}{36}\Omega^2 \left[\left(\left(\frac{6\beta_1^2}{b_1}+\frac{6\beta_2^2}{b_2}-2\right) +\frac{3\gamma}{\Omega\delta}\right)^2  +\left(\left(6b_2\beta_1-6b_1\beta_2-1\right)+\frac{3\gamma}{\Omega\delta}\right)^2\right] \\
     &=\frac{\delta^6}{36}\Omega^2 \left[ \left(\frac{6(\beta_1 -\frac{1}{2}b_1)^2}{b_1(1-b_1)}-\frac{1}{2}+ \frac{3\gamma}{\Omega\delta}\right)^2 + \left(6(\beta_1 -\frac{1}{2}b_1) -1 +\frac{3\gamma}{\Omega\delta}\right)^2\right].
\end{align*}
Assuming that $\frac{3\gamma}{\Omega\delta}\ll 1$, which holds for the underdamping Langevin dynamics when the damping term is small, we can find that the optimal parameter is $b_1^*=0.5$ and $\beta_1^* \approx 0.40 -\frac{0.43\gamma}{\Omega\delta}$. 
\end{remark}

\begin{remark}[Order of NySALT for the linear Langevin system] \label{rmk:orderSNO}
The local strong order of the NySALT in \eqref{eq:sto_Nys} is $O(\delta^{1.5})$. In fact, letting $\mathbf{X}_t^N = \mathbf{X}_t$ in \eqref{eq:Solu_stoc}--\eqref{eq:sto_Nys} and set $t=0$, we have
\begin{align*}
\E[|\mathbf{X}_{\delta} - \mathbf{X}_{\delta}^N|^2 ]^{1/2} 
&\leq \E[ \left \|  e^{A_\gamma\delta} - B_{b_1,\beta_1,\gamma}^\delta \right\|]\, \E[\|\mathbf{X}_0 \|^2]^{1/2}+ \left( \E[\left | \mathbf{W}^\delta_0 -\begin{pmatrix}
0\\ \xi_0^\delta \end{pmatrix} \right |^2 ]\right)^{1/2}. 
\end{align*}
The first term is of order $O(\delta^2)$, which follows from the above expansions.
 The second term is of order $O(\delta^{1.5})$ because with the notation $\Gamma =\bcm
0& 0\\ 0 & -\gamma 
\ecm $ and $A = \bcm
0& 1\\ -\Omega & 0 
\ecm$, 
 \[
\bm{W}_0^\delta - \begin{pmatrix}
0\\ \xi_0^\delta \end{pmatrix} = \sigma \int_{0}^{\delta} \left[ e^{A_\gamma(\delta - s)} -e^{ \Gamma(\delta -s)}
\right] 
\begin{pmatrix}
0\\ d\mW_s\end{pmatrix} 
= \sigma \int_{0}^{\delta} \left[ e^{A(\delta - s)} 
- I \right] 
e^{ \Gamma(\delta -s)} \begin{pmatrix}
0\\ d\mW_s\end{pmatrix},
\]
whose dominating component is $\sigma  \int_0^\delta A s e^{\Gamma s} d\mW_s $, a term with order $O(\delta^{1.5})$. 
Therefore, the local order of the NySALT scheme for the linear system is $O(\delta^{1.5})$. 
\end{remark}


\section{The benchmark problems: Fermi-Pasta-Ulam (FPU) model}\label{sec:numFPU}

In this section, we examine the performance of NySALT scheme on two benchmark nonlinear systems: Hamiltonian systems with the FPU potential (the deterministic FPU) and Langevin dynamics with the FPU potential (the stochastic FPU). Numerical results show that the inference is robust: the estimators are independent of the fine data generators, they converge as the number of trajectories increases, and they stabilizes fast (within a dozens of trajectories). NySALT scheme is efficient and accurate: it provides integrators adaptive to large time step size, improving the accuracy of solutions and enlarging the admissible time step size of stability, often quadruples those of the classical schemes, with minimum cost of training.

\subsection{The FPU system}\label{sec:FPU_def}
The FPU (Fermi-Pasta-Ulam) system \cite{fermi1955studies} presents highly oscillatory nonlinear dynamics. It 
consists of a chain of $2(m+1)$ mass points, connected with alternating soft nonlinear and stiff linear springs, and fixed at the end points \cite{Hairer2006}. The variables $q_1,\dots,q_{2m}$ (with $q_0=q_{2m+1}=0$) denote the displacements of the moving mass points, and $p_i$ denote their velocities.
The motion is described by a Hamiltonian system 
with the Hamiltonian $H$ given by 
\begin{equation}\label{eq:FPU-Hamiltonian}
H(\mbf{p},\mbf{q}) =K(\mbf{p})+V(\mbf{q})= \frac{1}{2}\sum_{i=1}^m (p_{2i-1}^2 + p_{2i}^2) + \frac{\omega^2}{4}\sum_{i=1}^m (q_{2i}-q_{2i-1})^2 + \sum_{i=0}^m (q_{2i+1}-q_{2i})^4. 
 \end{equation}
Here $\omega$ represents the stiffness of the system. We consider the system with $m=3$ and $\omega=50$. 

This nonlinear system is a benchmark problem for symplectic or quasi-symplectic integrators, which aim to produce stable and qualitatively correct simulations \cite{Hairer2006}. As discussed in Section \ref{sec:DetNstrom} and in Figure \ref{Fig:FPU_linear_stable}, the popular Str\"{o}mer--Verlet method can only tolerate a limited time step size when the system is stiff, otherwise the it leads to qualitatively incorrect energies. 
Here the quantities of interest are the energy of each stiff spring and their total stiff energy. More specifically, with a change of variables for $i=1,\dots, m$,

\begin{align} \label{eq:pos_stiff}
x_{0,i}:=\frac{q_{2i}+q_{2i-1}}{\sqrt{2}},\quad 
&  x_{1,i}:=\frac{q_{2i}-q_{2i-1}}{\sqrt{2}},\\ \label{eq:vel_stiff}
y_{0,i}:=\frac{p_{2i}+p_{2i-1}}{\sqrt{2}},\quad 
&  y_{1,i}:=\frac{p_{2i}-p_{2i-1}}{\sqrt{2}},
\end{align}
   where $x_{0,i}$ represents a scaled displacement of the $i$th stiff spring,
$x_{1,i}$ a scaled expansion (or compression) of the $i$th stiff spring, and $y_{0,i}$, $y_{1,i}$ their
velocities.  The total stiff energy and the energy of the $j$th stiff spring are
\begin{equation}\label{eq:stiff_energy}
 I:=\sum_{j=1}^{m}I_j, \text{ where }   I_j(x_{1,j},y_{1,j}):=
    \frac{1}{2}\left(y_{1,j}^2+\omega^2x_{1,j}^2\right),\quad j=1,\dots, m. 
\end{equation}

\paragraph{Properties of the deterministic FPU.} 
For a large $\omega$, the deterministic FPU model analytically exhibits behaviour dependent on initial data and time scales \cite{Hairer2006}. 

Depending on the initial condition, the system can present either close to linear or highly nonlinear dynamics. It behaves close to a linear system when the initial state is dominated by the stiff springs, that is when the total energy of stiff springs is of order $O(1)$ and the total energy of soft springs is less or of the same order. The system behaves nonlinearly when the initial state is mixed with both stiff and soft springs, which happens when the total energy of stiff springs is of order $O(1)$ and the total energy of soft springs is of order $O(\omega^2)$ or bigger. Numerical tests show that even trained only from one type of these initial conditions, the NySALT scheme can predict the dynamics of the other type of initial conditions. 

The FPU system shows dynamics varying with time scales as well. When the system starts from the nearly harmonic state (i.e., the first case of initial conditions), it will behave differently as the time evolves \cite{Hairer2006}: 
\begin{itemize}
    \item \textbf{short time scale $\omega^{-1}$.} The vibration of the stiff linear springs is nearly harmonic.
    \item \textbf{median time scale $\omega^{0}$.} This is the time scale of the motion of the soft nonlinear springs.
    \item \textbf{long time scale $\omega^{1}$.} Slow energy exchange among the stiff springs takes place on this time scale.
\end{itemize}
We will test the NySALT scheme \eqref{flow_map_deter} in these three time scales. 

\paragraph{Properties of the stochastic FPU.}

Stochastic perturbations can help simulate qualitatively the long time chaotic effect of the deterministic nonlinear model. Thus, 
Stochastic FPU models have been used to study the thermal conductivity and transport \cite{basile2006momentum,Roy2012,zhu2021effective}, asymptotic properties \cite{Schmid2020NonlinearityAT} and the stochastic resonance \cite{Miloshevich2009}. We consider a stochastic FPU with an additive white noise on the velocity and with a fraction. The noise injects energy while the friction dissipates energy, introducing random fluctuations to the energies. When they are relative small compared to the Hamiltonian, the stochastic FPU has dynamical properties similar to those of the deterministic system in terms of dependence on the initial data and the time scales. However, the total energy can fluctuate significantly larger than the total energy of the deterministic system, as we shown in Figure \ref{Fig:SFPU_gap200}. The stochastic FPU model is ergodic (see e.g.,\cite{MST10} and \cite[Proposition 6.1]{pavliotis2014stochastic}). Thus,  we will examine the NySALT and BAOAB schemes on producing the statistics of the energies, such as the time auto-covariance functions (ACF) and the empirical distributions (PDF).

\subsection{NySALT for the deterministic FPU}
\label{sec:detFPU_num}
We examine two aspects of the NySALT: the robustness of the inference and its numerical performance as an integrator for large time-stepping.
\paragraph{Numerical settings.} 
Unless otherwise specified, the numerical setting are as follows. 
We estimate the parameter $(b_1^*,\beta_1^*)$ from $M=100$ short trajectories on the training time interval $[0,T_{\text{tr}}]$ with $T_{\text{tr}}=1/2$ as described in Section~\ref{sec:flowDet}. Therefore, $T_{\text{tr}}$ is in the time scale $\omega^0$. The initial conditions are sampled according to
 \begin{equation}\label{init_conditions}
    \begin{split}
    &\text{soft spring: }x_{0,i}(0)= 1,\;\; y_{0,i}(0)=1,\\
     &\text{stiff spring: }x_{1,i}(0)= 1/\omega+\zeta_i,\;\; y_{1,i}(0)=1+\eta_i,
    \end{split}
 \end{equation}
where $\zeta_i$ and $\eta_i$ are independent Gaussian random variables with distribution $\frac{1}{\omega}\, \mathcal{N}(0,1)$.
 This initial distribution covers the regions that the entire system is nearly harmonic at the beginning of evolution. The data trajectories, recorded as time instants $t_n=n\delta$, are generated by the St\"{o}rmer--Verlet with a fine time step $h = 1\mathrm{e}{-4}$, except when testing the dependence on the symplectic integrator. The step size $\delta = \Gap \times h$ is much larger than $h$, and we will test $\Gap$ in several ranges. The optimal parameter is computed by constrained optimization with the interior point method with the loss function in \eqref{eq:loss_det}.

\paragraph{Robustness of the inference.} The NySALT depends on data by design. Thus, it will depend on the system generating data and its parameters converge as the data size increases. Importantly, it does not depend on the numerical integrator generating the accurate fine data for training. We examine them numerically below. 
\begin{itemize}[leftmargin=*]\setlength\itemsep{-1mm}
\item{Robustness to data generator.} We show first that NySALT is robust to the data generator. That is, the inferred parameter does not depend on the integrator generating the training data, as long as the integrator is accurate, which is realized by utilizing a sufficiently small time step $h$ and by using only short trajectories so that the accumulated numerical error is small. 

Table~\ref{Table:vary_fine_data} shows that  the estimated parameter are the same for three integrators, indicating the robustness of NySALT to the data generator. The three integrators are from the two-step Nystr{\"{o}}m family and one of them is the St{\"{o}}rmer–-Verlet method.  To ensure that numerical error in data is negligible, we use $h=1\mathrm{e}{-6}$. Since these integrators are second order $O(h^2)$ methods, their numerical error in the training interval $[0,T_{\text{tr}}]$ of order $O(10^{-12})$. The NySALT has time step $\delta = \Gap \times h$ with  $\Gap=\{1000;\;  5000; \; 10,000\}$, that is $\delta =\{0.001;\;  0.005;\;  0.01\}$.

\begin{table}[htp!]
    \centering
    \begin{tabular}{c|c|c|c}
    \hline
        {}&Parameter of data generator & Opt $\beta_1^*$ & Opt $b_1^*$  \\
        \hline
       $\Gap =1,000$ & $b_1^{\text{F}}=2/3$; $\beta_1^{\text{F}}=1/3$ & $0.403$ &$0.499$ \\
       {$\Gap =5,000$}& {} &  $0.403$ &$0.500$  \\
         {$\Gap =10,000$}& {}& $0.402$ &$0.499$  \\
          \hline
         {$\Gap =1,000$}& $b_1^{\text{F}}=1/3$; $\beta_1^{\text{F}}=1/3$ & $0.403$ &$0.499$\\
       {$\Gap =5,000$}& {} & $0.403$ &$0.500$ \\
         {$\Gap =10,000$}& {} & $0.402$ &$0.499$  \\
         \hline
         {$\Gap =1,000$}& $b_1^{\text{F}}=1/2$; $\beta_1^{\text{F}}=1/2$ & $0.403$ &$0.499$ \\
       {$\Gap =5,000$}& {} & $0.403$ &$0.500$ \\
         {$\Gap =10,000$}& {} & $0.402$ &$0.499$  \\
         \hline
    \end{tabular}
    \caption{Inferred parameters from data sets generated by three Nystr{\"{o}}m integrators with parameter $(b_1^{\text{F}},\,\beta_1^{\text{F}})$. The fine step size is $h=1\mathrm{e}{-6}$ and the training time is $T_{\text{tr}}=1/2$. The coarse step size is $\delta =\{0.001;\;  0.005;\;  0.01\}$, which corresponds to  $\Gap =\{1000;\; 5000;\;  10000\}$. 
    }
    \label{Table:vary_fine_data}
\end{table}
\item{Optimal parameters verse stiffness parameter $\omega$.} We examine next the dependence of the parameters on $\omega$, which determines the stiffness of the system.   
Here we test $\omega\in \{2,4,8, (10:10:100)\}$ with $M=100$. Since the linear stability of the St{\"{o}}rmer–-Verlet requires $\Delta t<\frac{2}{\omega}$, 
the coarse step is set to be $\delta=1/\omega$, which is half of the critical step size of linear stability. In comparison, we estimate the parameters of linear Hamiltonian systems \eqref{eq:Linear} with the same $\omega$ by minimizing $\|e^{A\delta}-B_{b_1,\beta_1}^{\delta}\|^2_{\Sigma^{-1}}$ in Proposition \ref{prop:linear_para}.  Figure~\ref{Fig:para_vs_omega} (left) shows that inferred parameters for FPU are close to those of the linear Hamiltonian system when $\omega\geq 30$. 

\begin{figure}[htp!]
    \centering
   \includegraphics[width = 0.44\textwidth]{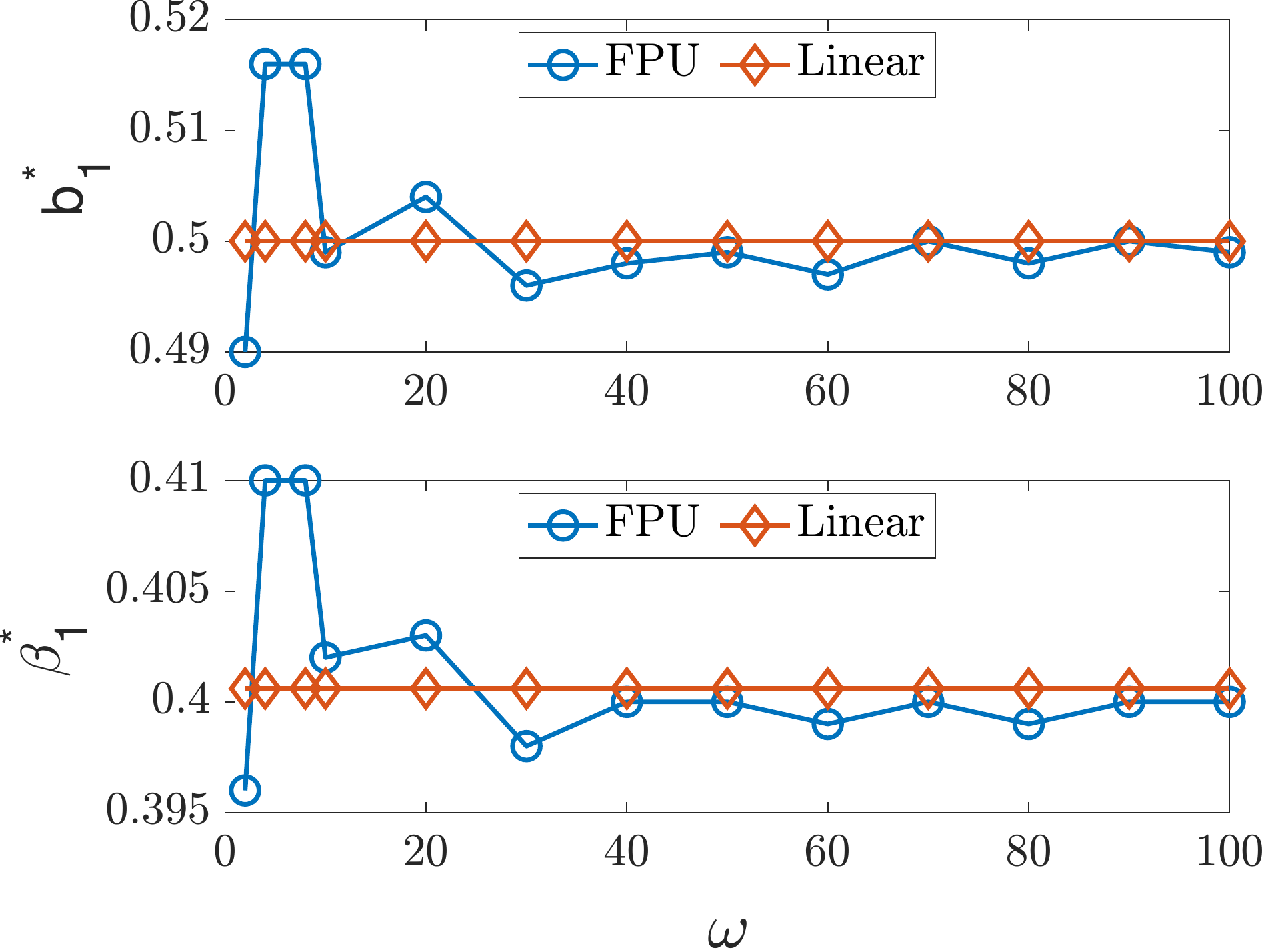}
   \includegraphics[width = 0.44\textwidth]{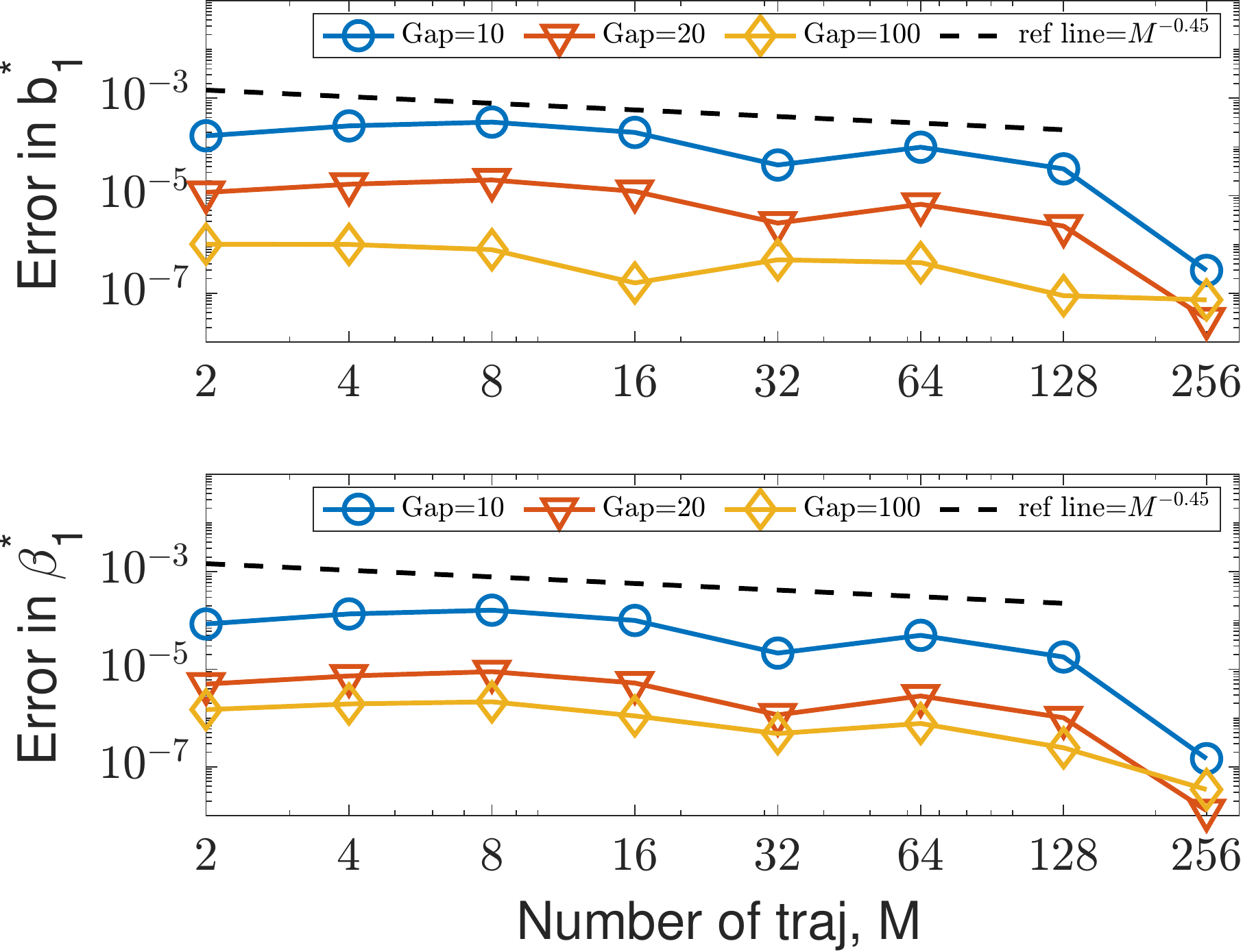} 
    \caption{ Robustness of the inference. {\bf Left:} Optimal parameters verse various stiffness values  $\omega$.  {\bf Right:} Convergence of estimators $(b_1^*, \, \beta_1^*)$ in numbers of sample trajectories $M$.    
    }
    \label{Fig:para_vs_omega}
\end{figure}

\item{Convergence in numbers of sample trajectories $M$.}
We examine next the convergence of the parameter estimator $\theta_M$ when the sample size $M$ increases with $\omega=50$.  Figure~\ref{Fig:para_vs_omega} (right) shows the error of the estimators with $M=2^{\{1:8\}}$ in comparison to the reference estimator with $M=2^9$. 
 As it can be seen, the estimator with $M=2$ is already close the reference values, with errors less than $10^{-3}$, and the error decays at a rate about $M^{-0.45}$, close to the theoretical rate in Theorem \ref{thm_convEst}. 
\end{itemize}

\paragraph{Numerical performance as an integrator for large stepping.} The NySALT provides an integrator adaptive to large time step size. Since it utilizes the optimal parameters adaptive to each step size, it improves the accuracy of the solution and enlarges the admissible time step size of accuracy, as the following numerical test demonstrates. 
\begin{itemize}[leftmargin=*]\setlength\itemsep{-1mm}
 \item {Improving the accuracy.}  Figure~\ref{Fig:accuracy_stab} (right) shows that the NySALT provides the most accurate solution for all time step sizes ranging from $\delta = \Gap \times h$ with $\Gap$ ranging in $\Gap\in \{(10:10:100),(150:50:350),390\}$, when comparing with the St\"{o}rmer--Verlet scheme. It presents the averaged relative Root-Mean-Square-Error (RMSE), which is averaged out over multiple trajectories, 
  \begin{equation}\label{avgRMSE_def}
     \text{Avg rel RMSE}:=\frac{1}{M}\sum_{m=1}^M \text{rel RMSE}^{(m)},
    \end{equation}
 where the relative RMSE in the $m$-th trajectory is defined as 
  \begin{align}\label{RMSE_def}
     &\text{rel RMSE}^{(m)}:=  \sqrt{\frac{1}{{N_t}}\sum_{i=1}^{N_t}\frac{\left(I^{\text{F},(m)}(t_i)- I^{ \text{C},(m)}(t_i)\right)^2}{\left(I^{\text{F},(m)}(t_i)\right)^2} }.
\end{align}
Here $I^{\mathrm{F},(m)}(t_i)$ denotes the total energy of the reference solution with  fine time step size  $h=1\mathrm{e}{-4}$ at time $t_i$ in the $m$-th trajectory, and similarly, $I^{\mathrm{C},(m)}(t_i)$ denotes the total energy of NySALT with  coarse time step $\delta= \Gap \times h$. The number of sample trajectories is $M=400$. 

We consider two time intervals $[0,T_{\text{test}}]$ with $T_{\text{test}}=0.5$ and $T_{\text{test}}=100$, representing the median and long time scale $O(\omega^0)$ and $O(\omega^1)$, respectively.  At $T_{\text{test}}=0.5$, as shown in top left of Figure~\ref{Fig:accuracy_stab}, both integrators have errors increasing linearly in $\delta = \Gap\times h$. The relative error by NySALT is  two magnitudes smaller than that by Verlet until around 
$\Gap=300$. The linear dependence with slope $2$ comes from the order $O(\delta^2)$ of the Nystr\"om methods. 
At $T_{\text{test}}=100$, as shown in bottom left of Figure~\ref{Fig:accuracy_stab}, the NySALT keeps the linear dependence of the relative error on $\delta=\Gap \times h$ up to $\Gap = 50$, doubling the reach of the Verlet method. 
Furthermore, up to the $\Gap =390$, the relative error of NySALT scheme is consistently smaller than that of the Verlet scheme. 
As a result and as we show next, NySALT can tolerate a larger time step size beyond the limitation of Verlet.

We further validate the accuracy of NySALT with a large time step by examining the transitions of energies in the long time scale $O(\omega^1)$. We consider sample $M=400$ trajectories with the time interval $[0,300]$. The right of 
Figure~\ref{Fig:accuracy_stab} shows that NySALT preserves the energy transition well, with errors significantly smaller than the suboptimal Nystr\"om integrator with parameters $b_1=0.45$ and $\beta_1=0.43$. The St{\"{o}}rmer--Verlet is not presented here because its errors are too large.   
Here we use the $L_1$ errors of the energies and phase angles to quantify the accuracy. The $L_1$ errors of the energies at time $t_i$ is computed as 
\begin{equation}\label{L1_tranErr}
    \text{Err}_{L^1}(t_i):=\frac{1}{I^{\text{F}}(t_i)} \sum_{j=1}^{3}\big|I_j^{\text{C}}(t_i)-I_j^{\text{F}}(t_i)\big|\cdot \delta 
\end{equation}
and the $L_1$ error of phase angles at time $t_i$
\begin{equation}\label{L1_angErr}
    \text{AngErr}_{L^1}(t_i):=\big|\vartheta^{\text{C}}(t_i) -\vartheta^{\text{F}}(t_i)\big|\cdot \delta+\big|\varphi^{\text{C}}(t_i)-\varphi^{\text{F}}(t_i)\big|\cdot \delta ,
\end{equation}
where the phase angles $\vartheta^{C}(t_i)$ and $\varphi^{C}(t_i)$ (and similarly $\vartheta^{F}$ and $\varphi^{F}$) are defined by
\[
\begin{split}
\vartheta^{C}(t_i):
=\mathsf{arccos}\left(\frac{\sqrt{I_3^{C}(t_i)} }{\sqrt{I^{C}(t_i)} }\right),
\quad 
\varphi^{C}(t_i)):
=\mathsf{arctan}\left(\frac{\sqrt{I_2^{C}(t_i)} }{\sqrt{I_1^{C}(t_i)} }\right). 
\end{split}
\]
The fine data for reference is generated by the St{\"{o}}rmer--Verlet with $h=1\mathrm{e}{-4}$. The coarse data are generated with $\delta =1/\omega$ (i.e., with $\Gap = 200$) by using the optimal and suboptimal parameters of Nystr\"{o}m methods.

\begin{figure}[htp!]
\centering
{\includegraphics[width = 0.44 \textwidth]{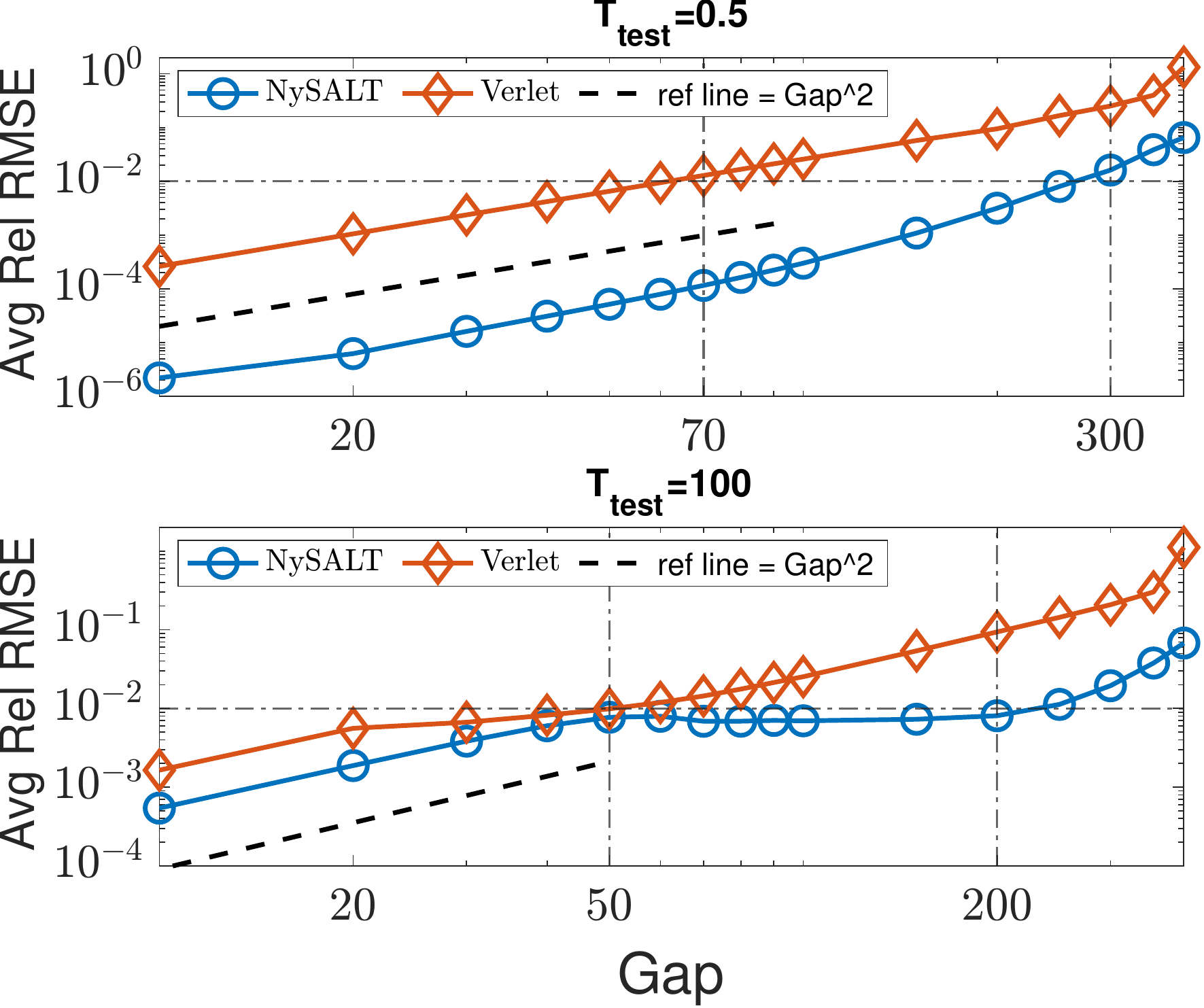}}
{\includegraphics[width = 0.44 \textwidth]{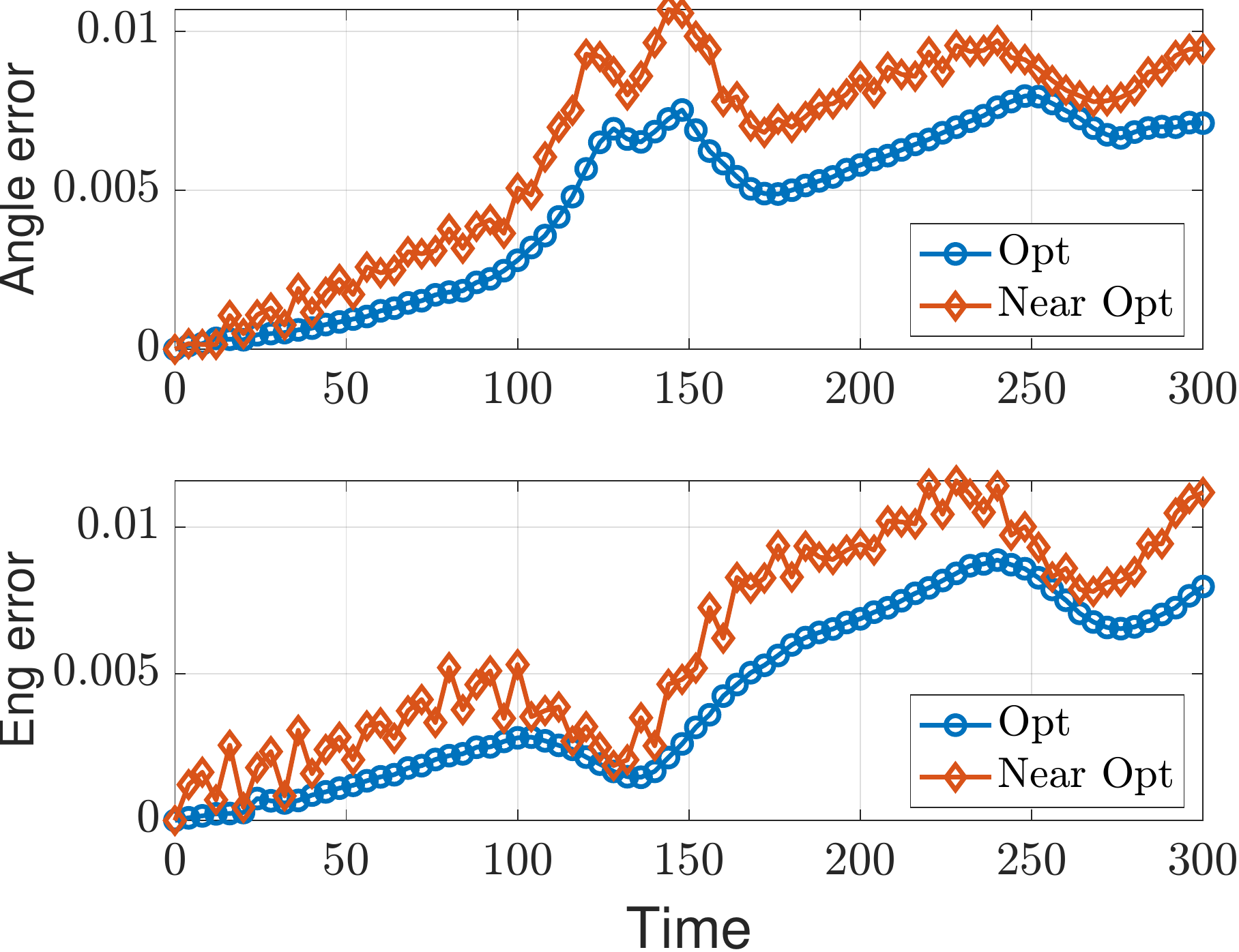}}
 \caption{Improving the accuracy. {\bf Left:} Averaged relative Root-Mean-Square-Error (Avg rel RMSE \eqref{avgRMSE_def} and \eqref{RMSE_def}) over total length $T_{\text{test}}$ between NySALT and St{\"{o}}rmer--Verlet schemes.  {\bf Right:} $L_1$ errors of the energies \eqref{L1_tranErr} and phase angles \eqref{L1_angErr} between NySALT and suboptimal Nystr\"{o}m schems. 
 }
    \label{Fig:accuracy_stab}
\end{figure}

\item{Enlarging the maximal admissible time step size.}
In  Figure~\ref{Fig:accuracy_stab} (top left) with the timescale of $O(\omega^{0})$, 
if we take threshold of 1\% average relative RMSE for $I$, the maximum gap in  St\"{o}rmer Verlet scheme allowed is 70, however, the maximum gap in NySALT scheme can reach at $\Gap=300$. Similarly in  Figure~\ref{Fig:accuracy_stab} (bottom left) with the timescale of $O(\omega^{1})$, the  maximum gaps allowed with 1\% average relative RMSE for both methods are 50 and 200. So NySALT scheme can enlarge at least four times of the maximal admissible step size of the St\"{o}rmer--Verlet scheme without lossing any accuracy. 

We demonstrate next that when $\delta=2/\omega$ (i.e., with $\Gap=400$), the linear stability limit of St\"ormer--Verlet,  NySLAT can remain stable while Verlet blows up.  
Figure~\ref{Fig:Fine_vs_coarse_DFPU} shows that the St{\"{o}}rmer--Verlet with coarse step size $\delta$ blows up almost immediately (within total time of 1), while the NySALT scheme remains stable and accurate and can capture the main patterns of the energy transfer up to total time of 150. Notice that the maximal admissible time step size of stability of the St{\"{o}}rmer--Verlet method is less than $2/\omega$, whereas NySALT can reach beyond it, reaching close (in additional tests) to $\sqrt{20/3}/\omega$,  which agrees the maximal admissible step size of linear stability in Remark \ref{rmk:linearStab}.

\begin{figure}[htp!]
    \centering
       \includegraphics[width =0.315\textwidth]{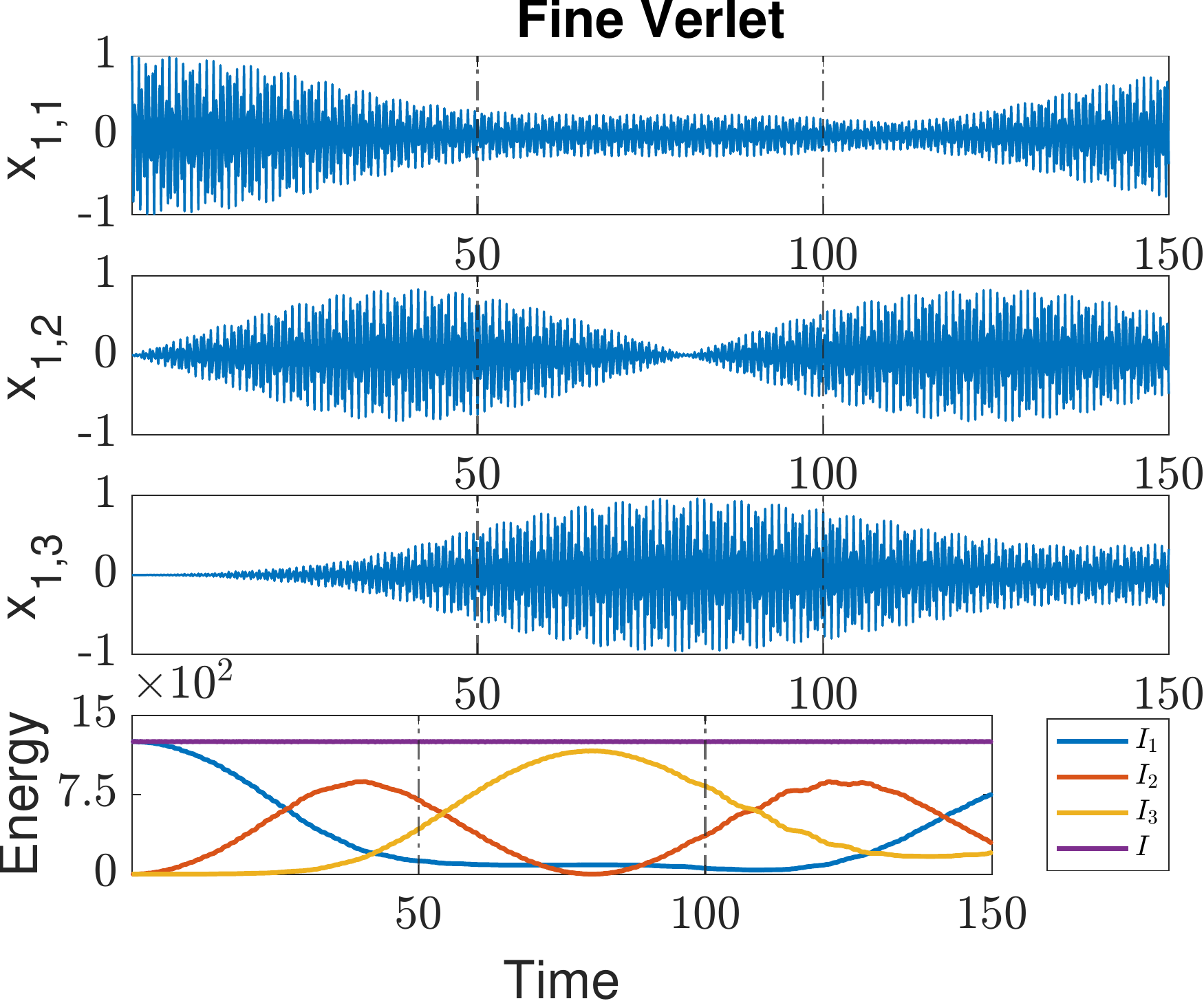}
      \includegraphics[width =0.3\textwidth]{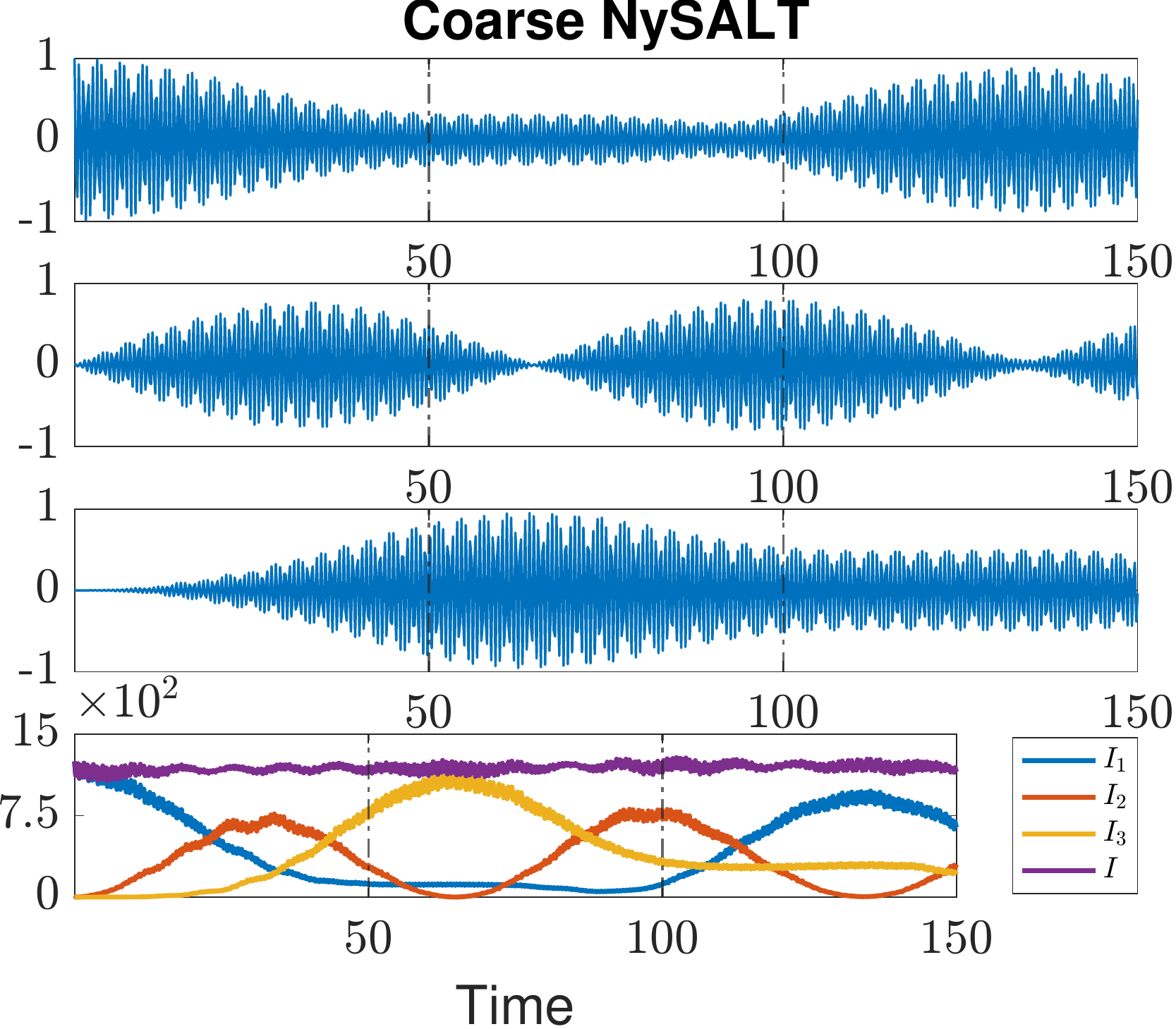}
      \includegraphics[width =0.3\textwidth]{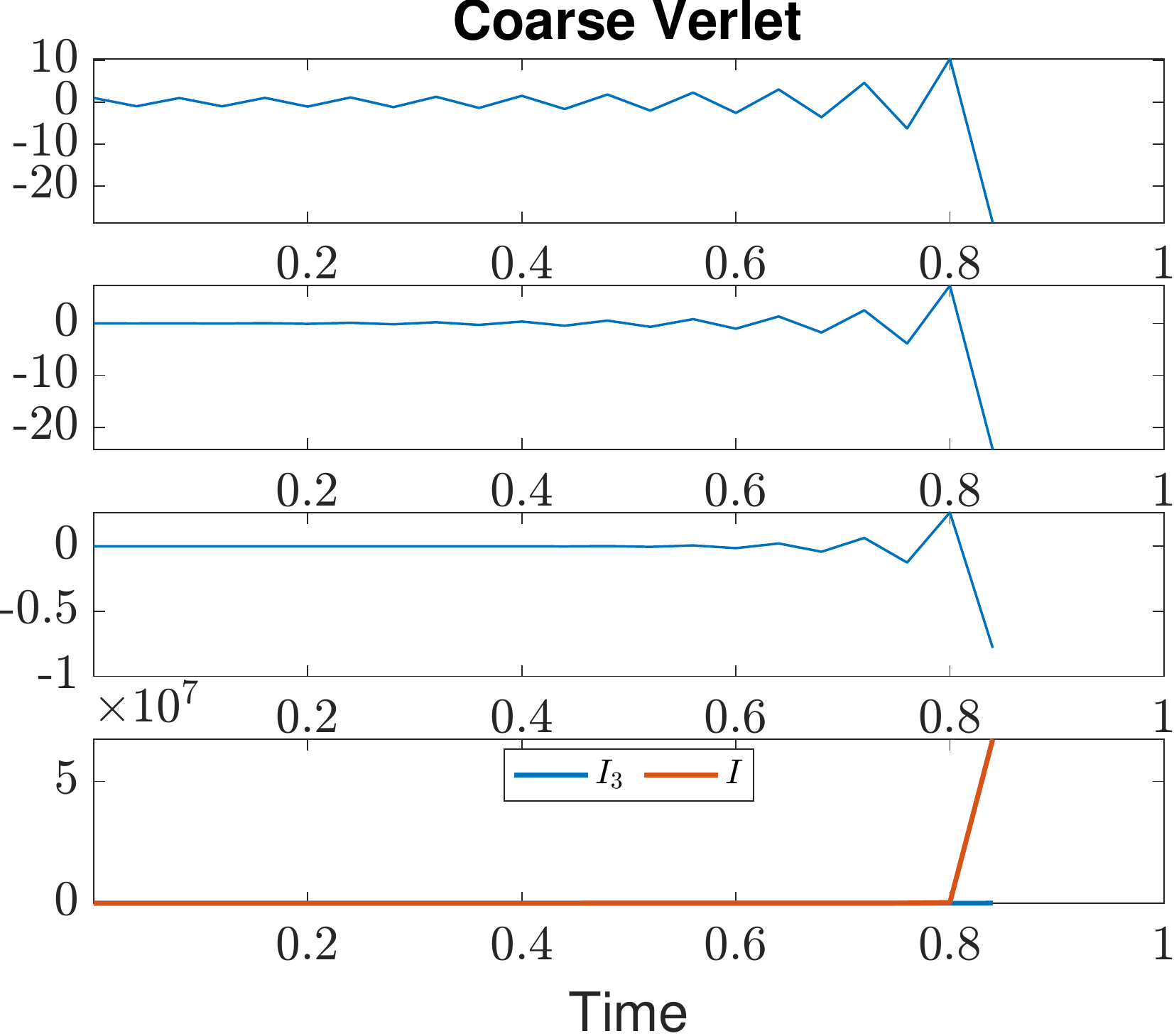}
    \caption{Large time-stepping near the linear stability limit. {\bf Left}, {\bf Middle} and {\bf Right} show the trajectories of scaled expansion of stiff springs $x_{1,i}$ \eqref{eq:pos_stiff} and stiff energies $I_i$ \eqref{eq:vel_stiff} generated by St{\"{o}}rmer--Verlet scheme with the fine step size $h$,  NySALT scheme with the coarse step size $\delta=400h$ and  St{\"{o}}rmer--Verlet scheme with the coarse step size $\delta=400h$. }
    \label{Fig:Fine_vs_coarse_DFPU}
\end{figure}

\end{itemize}

\subsection{NySALT for the stochastic FPU} \label{sec:stocFPU_num}
Similar to the deterministic example, we examine the stochastic NySALT scheme in terms of the robustness of its inference and its numerical performance as an integrator. 
\paragraph{Numerical settings.} 
We consider the Langevin dynamics with the same FPU potential  and the friction coefficient is $\gamma =0.01$, which is the underdamping case. The diffusion coefficient is $\sigma=0.05$. 
The optimal parameter $(b_1^*,\beta_1^*)$ are estimated by minimizing the loss function \eqref{eq:loss_sto} from $M= 512$ short trajectories on the training time interval $[0,T_{\text{tr}}]$ with $T_{\text{tr}}=1$ as described in Section~\ref{sec:flowSto}. In particular, the data trajectories consist of both the state $\mX_t$ and the stochastic force $\mW_t$ and they are generated by the BAOAB scheme with the fine time step size $h = 1\mathrm{e}{-4}$. We downsample the state trajectories at time instants  $t_n=n\delta$, and approximate the one-step stochastic increment at time instants $t_n$ by \eqref{eq:Xi_data}.  
The coarse time step size $\delta = \Gap \times h$ is much larger than $h$, with $\Gap$ ranging from 10 to 450. The initial conditions are uniformly sampled from a single long simulated trajectory of total time of $T=25000$.

\paragraph{Robustness of the inference.} The optimal estimators of NySALT scheme still stabilize very fast with small variations between different datasets. Figure \ref{fig:inf_sto} (Left) presents the absolute errors of the estimators with $M=2^{\{2:9\}}$, where the reference estimator is computed from $\overline{M} = 1024$ trajectories. 
 From the figure, both estimators with $M=512$ is close to the reference values, with absolute errors less than $10^{-2}$, and the error decays at rate about $M^{-0.44}$, which are close to theoretical rate in Theorem \ref{thm_convEst} as well.  Figure \ref{fig:inf_sto} (Right) further shows the mean and errorbar of the estimators at different time gaps. Both estimators are estimated with $M=512$ trajectories. The runtime analysis shows it takes about 854 seconds on average to learn the estimators at each gap. 
 We repeat the inference procedure independently 10 times to assess the variability over the random generated data. The mean of both estimators are close to optimal parameters in the linear Langevin system in sec.\ref{sec:lin-langevin}.
 Due to the small variance, the errorbar is barely noticeable. 

\begin{figure}[htp!]
    \centering
\includegraphics[width =0.45\textwidth]{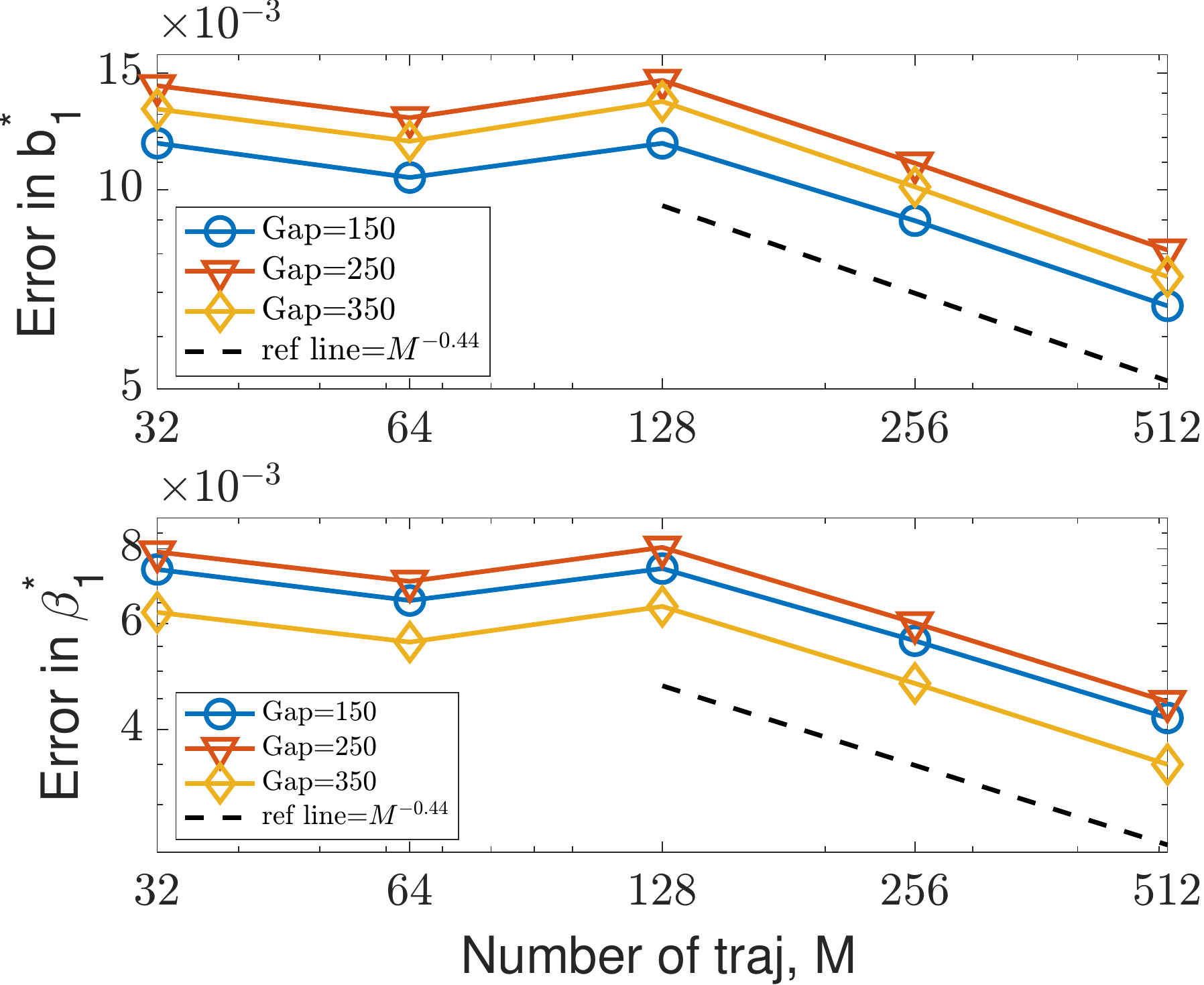}
\includegraphics[width =0.45\textwidth]{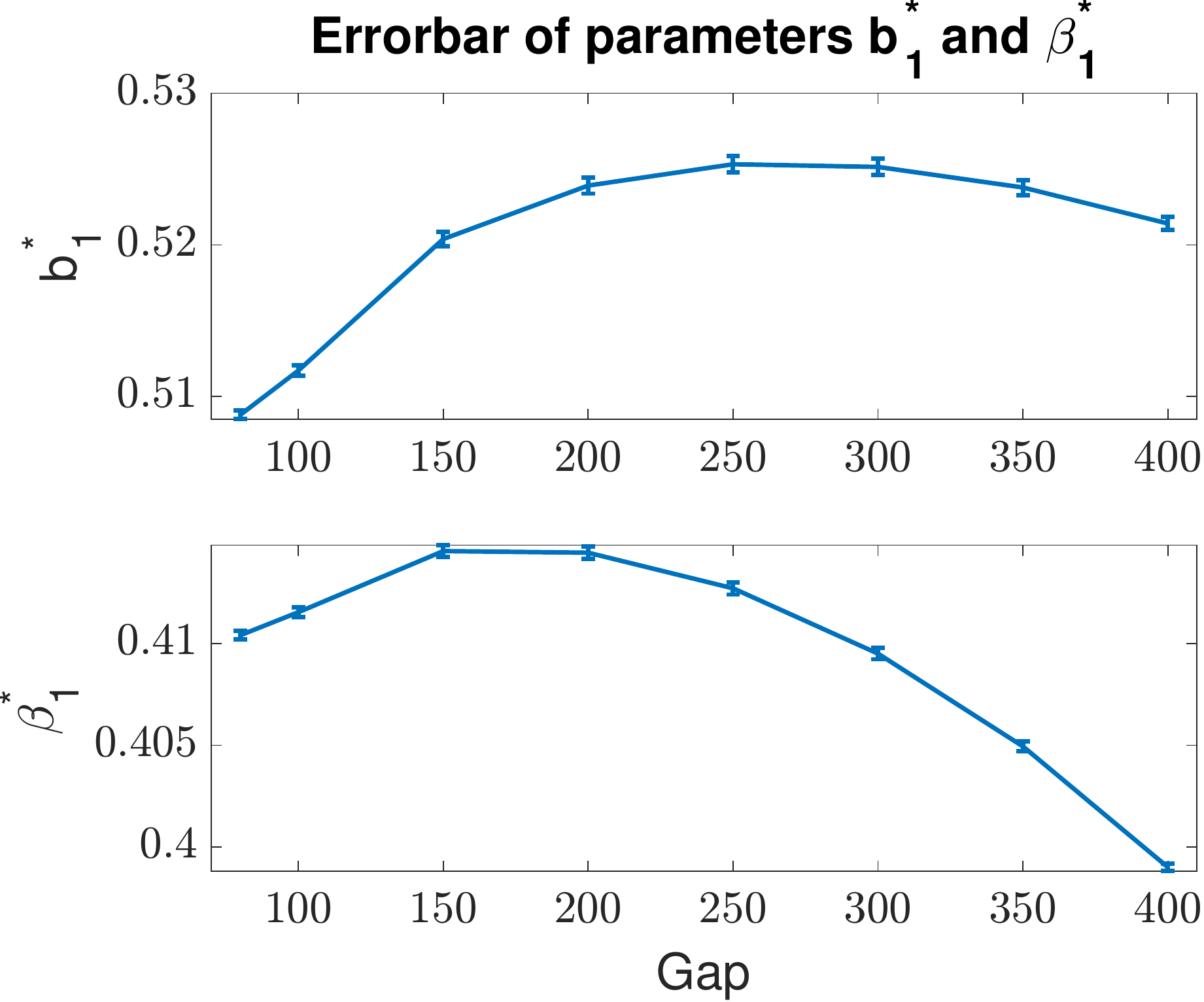}
   \caption{Robustness of estimators. \textbf{Left:} Convergence of parameters as number of trajectories $M$ increases.   \textbf{Right:} Mean and errorbar of estimators at different gaps in 10 independent simulations. 
    }
    \label{fig:inf_sto}
\end{figure}

 \begin{figure}[htp!]
    \centering
\includegraphics[width =0.32\textwidth]{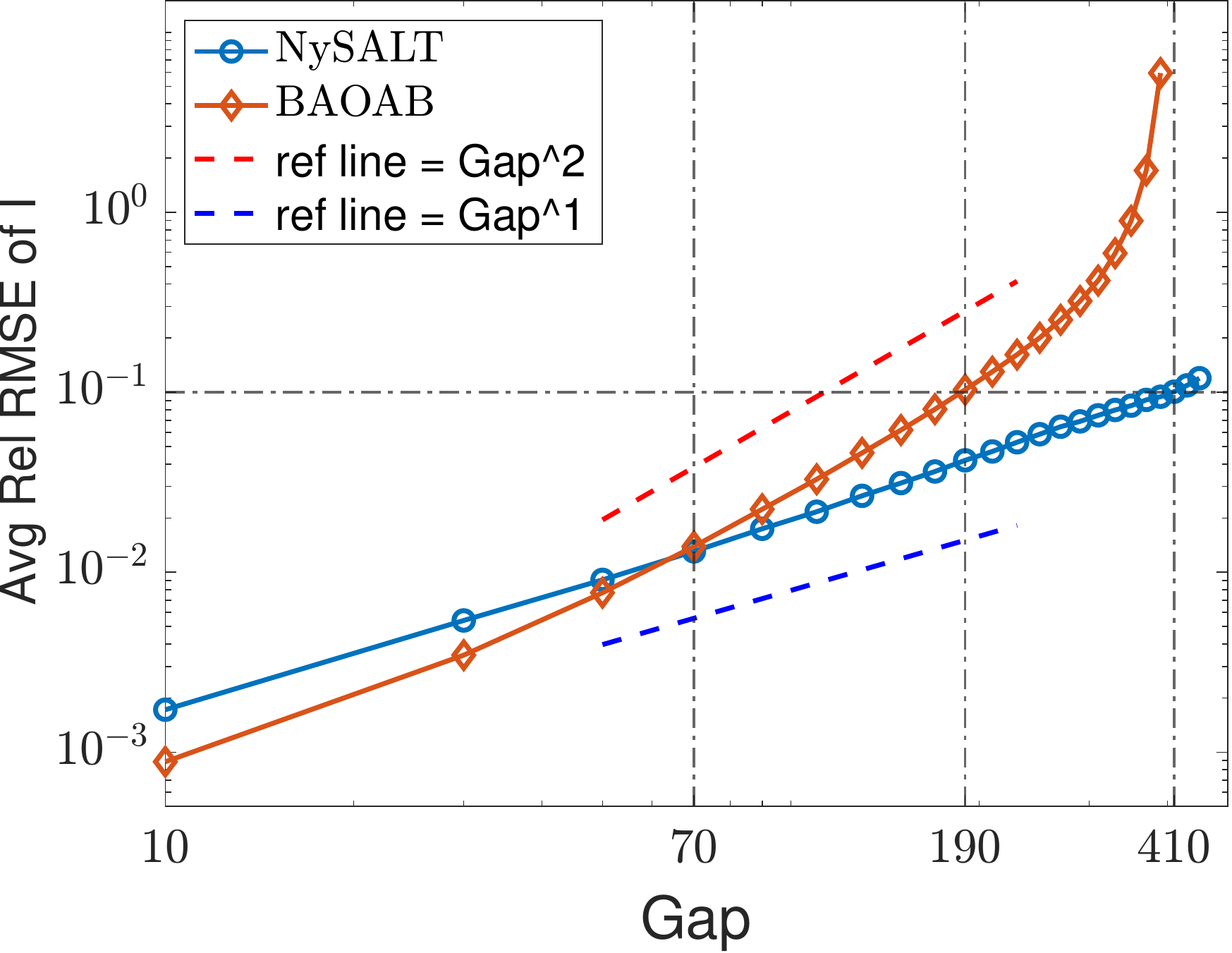}
\includegraphics[width =0.33\textwidth]{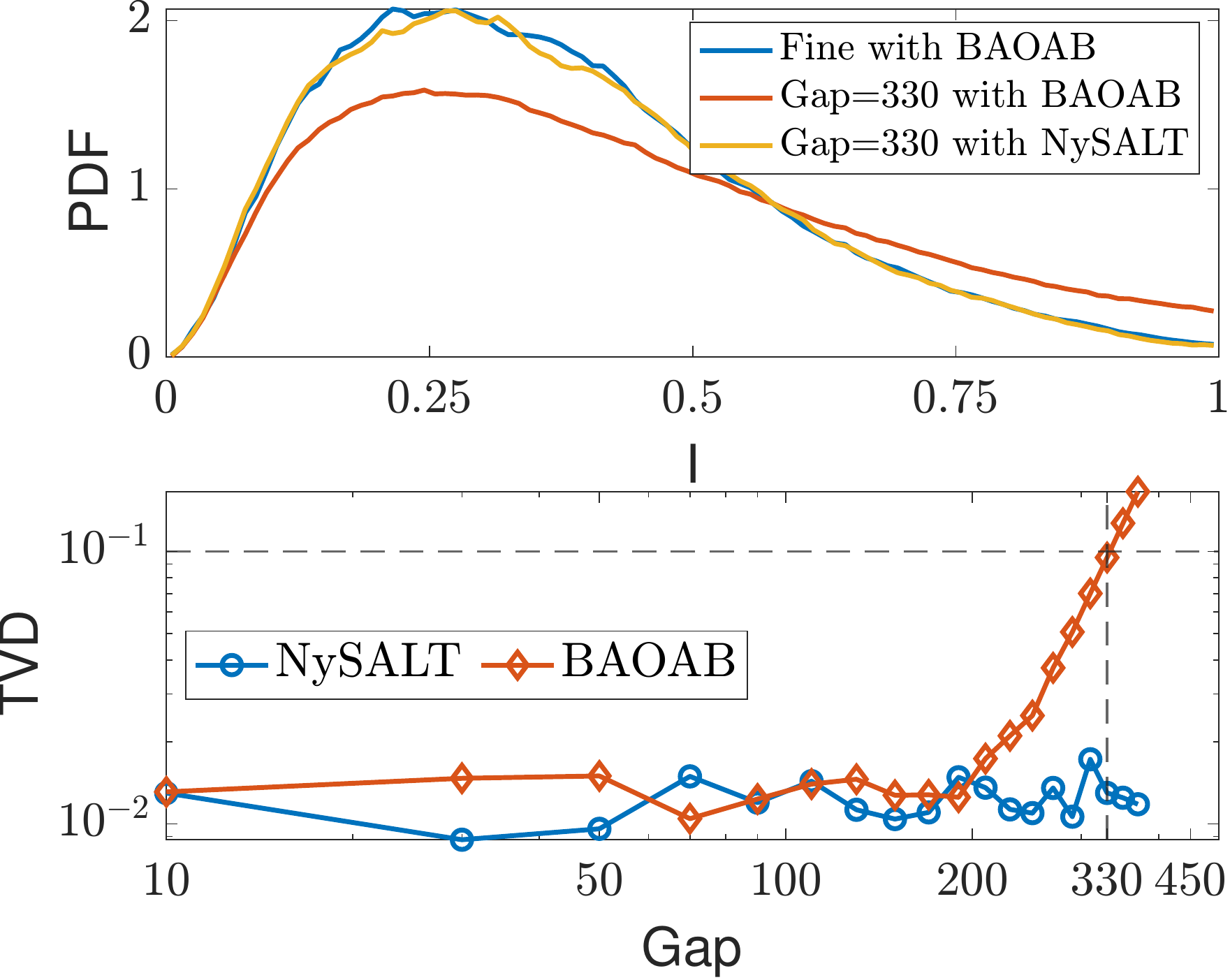}
\includegraphics[width =0.33\textwidth]{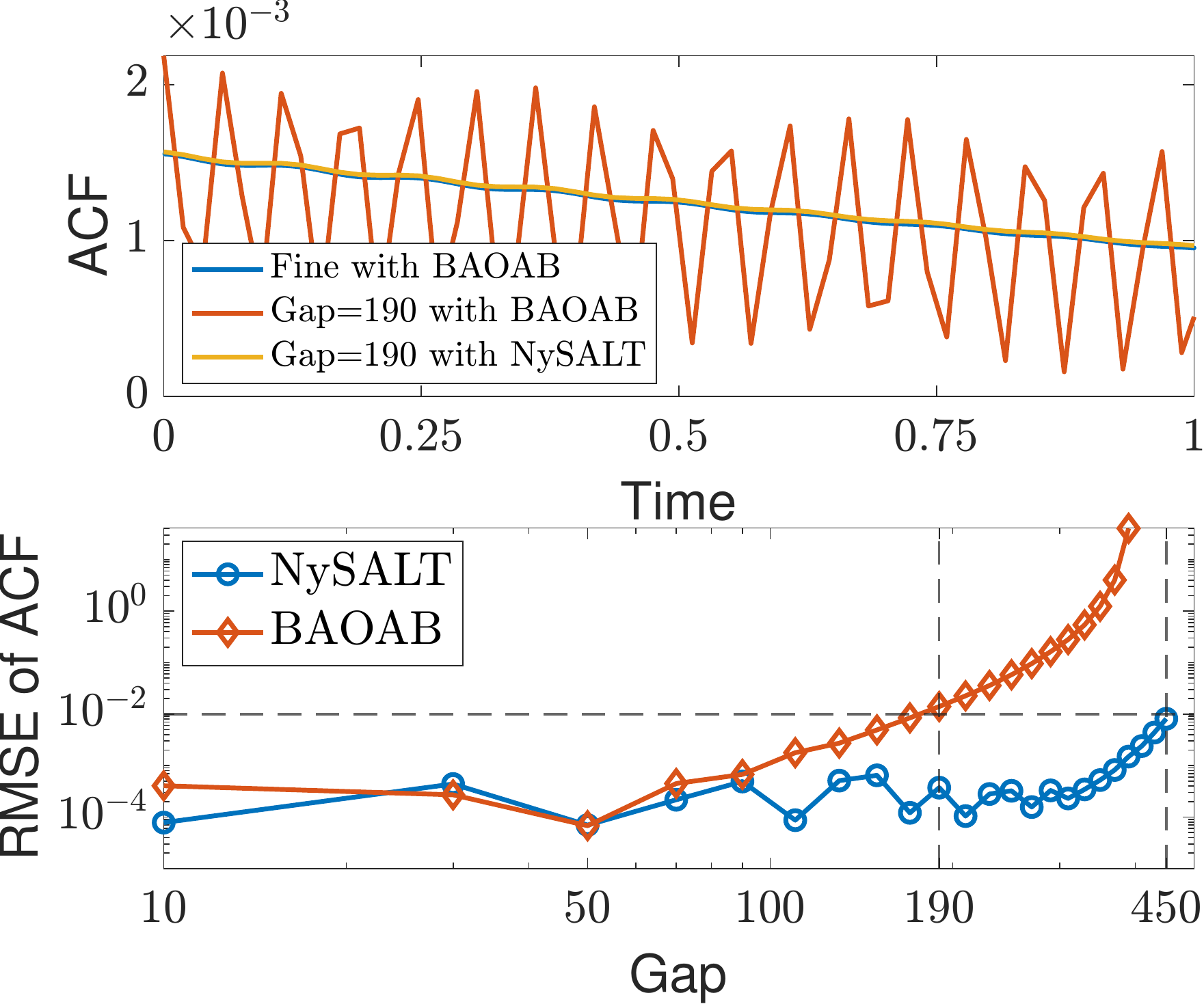}
   \caption{Performance of the NySALT scheme, in comparison with the BAOAB scheme. \textbf{Left:} Average relative RMSE (in \eqref{avgRMSE_def} and \eqref{RMSE_def}) of the total stiff energy $I$ over total length $T_{\text{test}}=1$. \textbf{Middle:} The empirical distributions (PDF) and their total variation distances of both schemes at various coarse time steps.  \textbf{Right:}  The time auto-covariance functions (ACF) and their RMSE of both schemes at various coarse time steps.
    }
    \label{fig:num_performance_sto}
\end{figure}

\paragraph{Numerical performance as an integrator.} The NySALT scheme has parameters adaptive to the time step size. Thus, like the deterministic FPU, it can tolerate relatively large time step when compared to a classical integrator, as verified by Figure \ref{fig:num_performance_sto}. Here we compare the NySALT scheme with BAOAB scheme, the state-of-the-art symplectic integrator in twofold: average relative RMSE in short time scale and statistics in long time scale. In the current setting, we only compare the results in terms of the total stiff energy $I$. All the parameters at various coarse time steps are estimated with $M=512$ sample trajectories.  

\begin{itemize}[leftmargin=*]\setlength\itemsep{-1mm}
\item{Average relative RMSE in short time scale.} We consider the time interval $[0, T_{\text{test}}]$ with $T_{\text{test}}=1$ and the number of sample trajectories $M=10000$. Both the NySALT and the BAOAB schemes integrate at the coarse time step $\delta=\Gap\times h$ for $\Gap$ ranging from 10 to 450, with the same coarse grained stochastic force $\xi_{t_i}$ generated by \eqref{eq:Xi_data} from $\mW_{t}$. Their solutions are compared with the reference solution generated by the BAOAB scheme with fine time step $h=10^{-4}$, with the same stochastic force $\mW_{t}$.  Figure \ref{fig:num_performance_sto} (Left) shows the average relative RMSE of the total stiff energy in both schemes, where average relative RMSE is defined in \eqref{avgRMSE_def} and \eqref{RMSE_def}. In log scale, the error of BAOBA scheme keeps the linear dependence until $\Gap=200$ with the slope 2, thereafter grows superlinearly.  However, NySALT scheme stretches the linear dependence to $\Gap=450$ with the slope 1. The error of NySALT scheme is consistently smaller than that of BAOAB after $\Gap=70$, which corroborates our goal of large time-stepping. If we take threshold of 10\% average relative RMSE, the maximum gap of BAOAB allowed is 70, while the maximum gap in NySALT can reach at $\Gap=190$.

\item{Statistics in long time scale.} Since the system is stochastic, we 
focus on statistics of long time trajectories, such as,  empirical distributions (PDF) and  auto-covariance functions (ACF). We consider the time interval $[0, T_{\text{test}}]$ with $T_{\text{test}}=40$ and the number of sample trajectories $M=10000$.  Similar to previous simulation, we integrate both schemes with the coarse time steps, whose gap ranging from 10 to 450. But the stochastic force in both schemes are not the same.  We estimate the PDF and ACF of the total stiff energy $I$ for different gaps. The empirical distribution is sampled with the 100 equal width bins in $[0,1]$ and 
ACF at time $\tau$ is defined
\begin{align} \label{eq:ACF}
\text{ACF}(\tau)=\mathbb{E}[I_t\bar{I}_{t+\tau}]-\mathbb{E}[I_t]\mathbb{E}[\bar{I}_{t+\tau}]
\end{align}
with $\tau\in[0,1]$.  
Similarly, these PDF and ACF are compared with the reference solutions generated by BAOAB scheme with fine step size.  We use the total variance distance (TVD) as the metric to quantify the deviation from the reference empirical measure. 
The total variance distance (TVD) between the empirical measure $P$ at coarse step size and the empirical measure $Q$ at fine step size is defined as
\begin{align}
    \text{TVD}(P,Q) = \frac{1}{2}\|P-Q\|_1.
\end{align}
On the other hand, we use the RMSE as the metric to measure the error of ACF. 

Figure \ref{fig:num_performance_sto} (Middle top) shows that at $\Gap=330$, NySALT scheme accurately reproduces the empirical distribution, whereas BAOAB scheme deviates largely from the reference due to the large time step.  
Figure \ref{fig:num_performance_sto} (Middle bottom)  shows that the NySALT scheme  has consistently smaller TVD than the BAOAB scheme after $\Gap>200$, remaining almost unchanged (around $10^{-2}$) even at $\Gap =450$. In particular, the TVD of BAOAB scheme at $\Gap=330$ is $10^{-1}$, which is one magnitude larger than that of NySALT.

Figure \ref{fig:num_performance_sto} (Right) shows the comparisons of the ACFs, with the top figure shows the ACFs when $\Gap=190$ and the bottom figure shows the RMSEs of the ACFs for both schemes with a ranges of gaps. The right top figure shows that at medium gap $\Gap= 190$, NySALT scheme produces an ACF almost exactly as the reference generated by BAOAB with  fine time step, in comparison,  BAOAB scheme with the same step size produces an ACF with significantly larger oscillations. Furthermore, the right bottom figure shows that the error of BAOAB scheme grows exponentially when $\Gap >100$, while NySALT scheme remains accurate until about $\Gap =400$. So the maximum admissible time step size of NySALT scheme almost quadruples that of BAOAB scheme. 

In addition, NySALT scheme significantly reduces the computational cost. For example, to compute the ACF by $M=10000$ sample trajectories, the run-time of the BAOAB scheme with fine step size is about 2078 seconds, whereas the NySALT with medium step size $\Gap=190$ only takes 18 seconds. It is almost 115 times faster. Even we take into account of the training time (which is about 854 seconds), it is still significantly better to use NySALT scheme. 
\end{itemize}

\section{Conclusion}
We have proposed and investigated a parametric inference approach to \emph{innovate classical numerical integrators} to obtain a new integrator which is tailored for each time step size and the specific system. 
In particular, we introduce NySALT, a Nystr\"{o}m-type inference-based schemes adaptive to large time-stepping. The framework of constructing inference-based schemes from data has the major advantages:   \vspace{-1mm}
\begin{itemize}[leftmargin=*]\setlength\itemsep{-1mm}
\item Compared to the generic classical numerical integrators, the inferred scheme with optimal parameters enlarges the maximal admissible while maintaining similar levels of accuracy. 
\item The parametric inference is robust regardless data generation and is immune to curse-of-dimensionality or overfitting. Moreover, the scheme is generalizable beyond the training set for autonomous systems. 
\item The convergence of the estimators can be rigorously proved as data size increases.
\end{itemize}
We demonstrate the performance of the NySALT on both Hamiltonian and Langevin system via the Fermi-Pasta-Ulam (FPU) potential. Numerical results verify the convergence of the estimators. Furthermore, they show that NySALT quadruples the time step size for the Hamiltonian system and quadruples that for the Langevin system when comparing with the St\"{o}rmer--Verlet and the BAOAB to maintain the average relative RMSE within certain level.

Meanwhile, NySALT scheme still has a limited maximal time step size, which is inherited from the classical integrator. The whole idea of NySALT can be easily extended to other family of the integrators. 
In the future work, we will investigate improved approximation of the flow map by using new parametric forms or non-parametric learning to further extend time step size.

\paragraph{Acknowledgements} X. Li’s is grateful for partial support by the National Science Foundation Award DMS-1847770. F. Lu's is grateful for partial support by the NSF Award DMS-1913243. M. Tao is grateful for partial support by the NSF DMS-1847802, NSF ECCS-1936776, and the Cullen-Peck Scholar Award. F. Ye is grateful for partial support by the AMS-Simons travel grants. 

\appendix

\section{Derivative of the cost function }\label{appendix-A}
We provide here the detailed computation of the derivative of the cost function in \eqref{eq:loss_sto}. 
Recall that with a given time step $h$, the Stochastic Symplectic Nystr\"{o}m scheme consists of two components: a symplectic 2-step Nystr\"{o}m scheme that integrates the Hamiltonian part:
\[
\tilde{\mX}_{n+1} = S^h_{b_1,\beta_1}\left(\mX_n\right),  
\]
and an exact integration of the Ornsterin-Uhlenbeck process: 
\[ 
\mX_{n+1}^N = O^h\tilde{\mX}_{n+1} +\bcm \bm{0}\\ \xi_n \ecm
\]
where $O^h = 
\bcm
\mbf{I} & \bm{0} \\
\bm{0} & \exp(-\gamma h )\mbf{I}
\ecm$.

The cost function is rewritten trajectory-wise as 
\begin{align*}
    \mathcal{E}_M(\theta) = \frac{1}{M}\sum_{m=1}^M \mathcal{E}_m
\end{align*}
where each summand function (superscript $m$ is omitted) is 
\begin{align*}
    \mathcal{E}_m(\theta)  &= \frac{1}{N_t} \sum_{i=0}^{N_t-1}\left\|\delta(F_\theta(\mX_{t_{i}},\xi_{t_i},\delta) - \mathcal{F}(\mX_{t_{i}},\mW_{[t_i,t_{i+1}]},\delta))\right\|^2_{\Sigma^{-1}} \\
    &= \frac{1}{N_t} \sum_{i=0}^{N_t-1}\left\|\mX^N_{t_{i+1}} - \mX_{t_{i+1}}\right\|^2_{\Sigma^{-1}}.
\end{align*}
Then, to compute the derivative of the cost function, it suffices to compute the derivative of each summand function, which is  
\begin{align*}
    \nabla_\theta \mathcal{E}_m(\theta)  &=  \frac{1}{N_t} \sum_{n=0}^{N_t-1} \left([\nabla_\theta(\mX_{t_{i+1}}^N)]^\top \Sigma^{-1} \left(\mX_{t_{i+1}}^N - \mX_{t_{i+1}} \right)\right) \\
    &=\frac{1}{N_t} \sum_{n=0}^{N_t-1} 2\left( \left(O^h\cdot J(b_1, \beta_1) \right)^\top\Sigma^{-1} \left(\mX_{t_{i+1}}^N - \mX_{t_{i+1}} \right)\right).
\end{align*}
Here $J(b_1, \beta_1)$ is the Jacobian of the symplectic integrator with respect to the parameters. It is computed directly as (recall the definition of $S^h_{b_1,\beta_1}$ in \eqref{eq:l1l2n} and $l_i$ in  \eqref{eq:nystrom2}) 
\begin{align*}
J(b_1, \beta_1) &= 
\bcm
\frac{\partial \tilde{\mbf{q}}_{n+1}}{ \partial b_1} &   \frac{\partial \tilde{\mbf{q}}_{n+1}}{ \partial \beta_1}  \\ 
\frac{\partial \tilde{\mbf{p}}_{n+1}}{ \partial b_1} &   \frac{\partial \tilde{\mbf{p}}_{n+1}}{ \partial \beta_1}
\ecm  \\
& = 
\bcm
 h^2\left(\beta_1\frac{\partial \ell_1}{\partial b_1}+\beta_2\frac{\partial \ell_2}{\partial b_1}\right)   &  h^2\left(\ell_1-\ell_2 + \beta_1 \frac{\partial \ell_1}{\partial \beta_1}+\beta_2\frac{\partial \ell_2}{\partial \beta_1} \right)   \\
 h\left(\ell_1 - \ell_2 + b_1 \frac{\partial \ell_1}{\partial b_1} + b_2\frac{\partial \ell_2}{\partial b_1} \right) & h\left(b_1\frac{\partial \ell_1}{\partial \beta_1} + b_2\frac{\partial \ell_2}{\partial \beta_1}\right)
\ecm
\end{align*}
where $\frac{\partial \ell_1}{\partial b_1} = \nabla\ell_1 \frac{\beta_1 h \mbf{p}_n }{b_1^2}$, $\frac{\partial \ell_1}{\partial \beta_1}  = -\nabla\ell_1 \frac{h \mbf{p}_n }{b_1}$, and 
\begin{align*}
\frac{\partial \ell_2}{\partial b_1} &= \nabla\ell_2\cdot \left( -h \mbf{p}_n \frac{\beta_2}{b_2^2} + h^2\frac{\beta_1b_2 - b_1\beta_2}{b_2} \frac{\partial \ell_1}{\partial b_1} -h^2 \ell_1 \frac{\beta_2}{b_2^2} \right) \\ 
& = \nabla\ell_2\cdot \left( -h \mbf{p}_n \frac{\beta_2}{b_2^2} + h^2\frac{\beta_1b_2 - b_1\beta_2}{b_2} \frac{\beta_1 h }{b_1^2}  \nabla\ell_1\cdot \mbf{p}_n - h^2 \ell_1 \frac{\beta_2}{b_2^2} \right), \\
\frac{\partial \ell_2}{\partial \beta_1} & = \nabla\ell_2 \cdot\left( h \mbf{p}_n \frac{1}{b_2} + h^2\frac{\beta_1b_2 - b_1\beta_2}{b_2} \frac{\partial \ell_1}{\partial \beta_1} + h^2 \ell_1 \frac{1}{b_2} \right) \\
& =\nabla\ell_2 \cdot \left( h \mbf{p}_n \frac{1}{b_2} - h^2\frac{\beta_1b_2 - b_1\beta_2}{b_2} \frac{h  }{b_1}  \nabla\ell_1 \cdot \mbf{p}_n + h^2 \ell_1 \frac{1}{b_2} \right). 
\end{align*}
Here $\nabla \ell_1$ and $\nabla \ell_2$ are
\begin{align*}
\begin{split}
&\nabla\ell_1 = \nabla g(\mbf{q}_n + hc_1 \mbf{p}_n), \\
&\nabla \ell_2 = \nabla g(\mbf{q}_n + hc_2 \mbf{p}_n + h^2a_{21}\ell_1).
\end{split}
\end{align*}

\bibliographystyle{plain}
{\small
\bibliography{ref_LearnDiscretization,ref_FeiLU2022_4}
}

\end{document}